\documentclass{amsart}

\usepackage[T1]{fontenc}

\usepackage{amssymb,amscd}
\usepackage{url}

\usepackage{amsthm,amsfonts,amsmath,amsxtra}

\usepackage[pdfborder={0 0 0 [0 0 ]},bookmarksdepth=5, bookmarksopen=true]{hyperref}
\usepackage{fancyhdr}

\hypersetup{
    colorlinks=true,
    linkcolor=black,
    citecolor=black,
    filecolor=black,
    urlcolor=black,
}

\def\writefig#1 #2 #3 {\rlap{\kern #1 truecm
\raise #2 truecm \hbox{\protect{\small #3}}}}

\usepackage{color}

\DeclareMathAlphabet{\mathbbu}{U}{bbold}{m}{n}

\usepackage{graphicx}
\usepackage{stmaryrd}
\usepackage[all]{xy}
\SelectTips{cm}{}   

\newtheoremstyle{bold}
{.5\baselineskip}{.5\baselineskip}{\itshape}{}{\bfseries}{.}{.5em}{}
\newtheoremstyle{shy}
{.5\baselineskip}{.5\baselineskip}{}{}{\bfseries}{.}{.5em}{}

\makeatletter
\def\@captionfont{\small}

\makeatother

\renewcommand{\ge}{\geqslant}
\renewcommand{\geq}{\geqslant}
\renewcommand{\le}{\leqslant}
\renewcommand{\leq}{\leqslant}

\theoremstyle{bold}
\newtheorem{theorem}{Theorem}[section]
\newtheorem{proposition}[theorem]{Proposition}
\newtheorem{lemma}[theorem]{Lemma}
\newtheorem{corollary}[theorem]{Corollary}
\newtheorem{conjecture}[theorem]{Conjecture}
\newtheorem{fact}[theorem]{Fact}

\theoremstyle{shy}
\newtheorem{definition}[theorem]{Definition}
\newtheorem{remark}[theorem]{Remark}
\newtheorem{example}[theorem]{Example}
\newtheorem{observation}[theorem]{\textbf{Observation}}  

\DeclareMathOperator{\tr}{tr}

\newcommand{\cA}{\mathcal{A}}
\newcommand{\cB}{\mathcal{B}}
\newcommand{\cC}{\mathcal{C}}
\newcommand{\cD}{\mathcal{D}}

\newcommand{\cF}{\mathcal{F}}

\newcommand{\cP}{\mathcal{P}}

\newcommand{\cR}{\mathcal{R}}

\newcommand{\cU}{\ts\ts\mathcal{U}}

\newcommand{\cW}{\mathcal{W}}

\newcommand{\FF}{\mathbb{F}}
\newcommand{\NN}{\mathbb{N}}

\newcommand{\ZZ}{\mathbb{Z}\ts}

\newcommand{\RR}{\mathbb{R}}
\newcommand{\CC}{\mathbb{C}}

\newcommand{\SSS}{\mathbb{S}}
\newcommand{\XX}{\mathbb{X}}

\newcommand{\ts}{\hspace{0.5pt}}

\title{Non-crossing partitions}

\author[B.~Baumeister]{Barbara Baumeister}
\author[K-U.~Bux]{Kai-Uwe Bux}
\author[F.~G{\"o}tze]{Friedrich G\"{o}tze}
\author[D.~Kielak]{Dawid Kielak}
\author[H.~Krause]{Henning Krause}

\begin{document}

\begin{abstract}
Non-crossing partitions have been a staple in combinatorics for quite some
  time. More recently, they have surfaced (sometimes unexpectedly) in various
  other contexts from free probability to classifying spaces of braid groups.
  Also, analogues of the non-crossing partition lattice have been introduced.
  Here, the classical non-crossing partitions are associated to Coxeter and
  Artin groups of type $\mathsf{A}_n$, which explains the tight connection
  to the symmetric groups and braid groups. We shall outline those
  developments.
  \end{abstract}

  \maketitle
  
\section{The poset of non-crossing partitions}
  \index{poset!non-crossing partition lattice}
  \index{lattice!non-crossing partition lattice}
  A \emph{partition}\index{partition} $p$ of a set $U$ is a
  decomposition of $U$ into pairwise disjoint subsets
  $B_{i}$:
  \[
    U = \biguplus_{i} B_{i} 
  \]
  The subsets $B_{i}$ are called the \emph{blocks}\index{partition!block}
  of the partition $p$. Another way to look at this is
  to consider $p$ as an equivalence relation on $U$.
  In this perspective, the subsets $B_{i}$ are the
  equivalence classes.
  Let $q$ be another partition of the same set $U$.
  We say that $q$ is a \emph{refinement}\index{partition!refinement} of $p$
  if each block of $q$ is contained in a block of $p$.
  In terms of equivalence relations, if two elements of $U$ are
  $q$-equivalent, they are also $p$-equivalent.
  We also say that $q$ is \emph{finer} than $p$
  or that $p$ is \emph{coarser} than $q$; and we
  write
  \(
    q\preceq p
  \).

  Let $\operatorname{P}(U)$ be the set of all partitions on the
  underlying set $U$.  The refinement relation $\preceq$ is a partial
  order on the set $\operatorname{P}(U)$, which is therefore a
  \emph{poset}\index{poset}.  Moreover, it is a
  \emph{lattice}\index{lattice}, i.e., every non-empty finite subset
  $\cP\subseteq\operatorname{P}(U)$ has a least upper bound and a
  greatest lower bound. We remark that the partition lattice is
  \emph{complete}, i.e., even arbitrary infinite subsets have least
  upper and greatest lower bounds.

  \begin{remark}\label{a4-c3-c13-half-a-lattice-is-a-lattice}
    It is interesting that the definition of a complete lattice can be
    weakened by breaking the symmetry between upper and lower
    bounds. If a poset has upper bounds and greatest lower bounds, it
    is already a complete lattice (i.e. it also has lowest upper
    bounds).
  \end{remark}
  \begin{proof}[Sketch of proof]
    Let $\cP$ be a non-empty subset of the poset. We consider the the
    set \( B^+(\cP) \)
    of all common upper bounds for the non-empty subset $\cP$. Since
    the poset has upper bounds, $B^+(\cP)$ is non-empty. Hence it has
    a greatest lower bound, which turns out to be the lowest upper
    bound of $\cP$.
  \end{proof}
  
  Consider the following reflexive and symmetric relations on $U$:
  \begin{align*}
    x &\sim y
      \quad:\Leftrightarrow\quad
      \exists p\in\cP\,:\,\,
        x \text{\ and\ } y \text{\ are\ }p\text{-equivalent}
      \\
    x &\approx y
      \quad:\Leftrightarrow\quad
      \forall p\in\cP\,:\,\,
        x \text{\ and\ } y \text{\ are\ }p\text{-equivalent}
  \end{align*}
  It is clear that $\approx$ is itself an equivalence relation.
  It corresponds
  to the \emph{meet}\index{lattice!meet} $\bigwedge \cP$
  of the partitions in $\cP$, i.e., the greatest lower
  bound of $\cP$. The transitive closure of $\sim$ is an equivalence
  relation, which corresponds to the \emph{join}\index{lattice!join}
  $\bigvee \cP$ of the partitions in $\cP$.
  
  Now, we restrict our consideration to finite sets.  For a natural
  number $m\in\NN$, let us denote by $[m]$ the set
  $\{\,1,2,3,\ldots,m\,\}$.  We fix the natural cyclic ordering on
  $[m]$ and represent its elements as the vertices
  \( v_{1},\ldots,v_{m} \)
  of a regular $m$-gon inscribed in the unit circle. Let $p$ be a
  partition of $[m]$. We say that two blocks $B$ and $B'$ of the
  partition $p$ \emph{cross} if their convex hulls intersect. The
  partition $p$ is called
  \emph{non-crossing}\index{partition!non-crossing partition} if its
  blocks pairwise do not cross. A non-crossing partition can thus be
  depicted by colouring the convex hulls of its blocks. For blocks of
  size one or two, we fatten up the convex hull.
  \begin{figure}[h]
    \begin{center}
      \raisebox{0pt}{\includegraphics{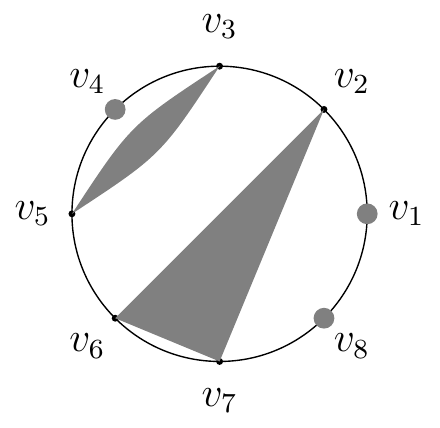}}

      \caption{Visualization of the partition     
      \(
        \{\,
          \{\,1\,\},
          \{\,2,6,7\,\},
          \{\,3,5\,\},
          \{\,4\,\},
          \{\,8\,\}
        \,\}
      \).\label{a4-c3-c13-fig-visualization}}
    \end{center}
  \end{figure}
  It is clear from the visualisation that the complements of the
  coloured regions also are pairwise disjoint. This gives rise to the
  \emph{Kreweras complement}\index{partition!non-crossing
    partition!Kreweras complement}.  Here, we put \emph{dual vertices}
  $w_{1},\ldots,w_{m}$ within the arcs $v_{i}-v_{i+1}$. There is no
  natural numbering, and we choose to place $w_{1}$ within the arc
  from $v_{1}$ to $v_{2}$.  Let $p$ be a non-crossing partition. Two
  dual vertices lie in the same block of the complement $p^\mathrm{c}$
  if they lie within the same complementary region of the convex hulls
  of blocks of $p$.
  \begin{figure}[h]
    \begin{center}
      \raisebox{0pt}{\includegraphics{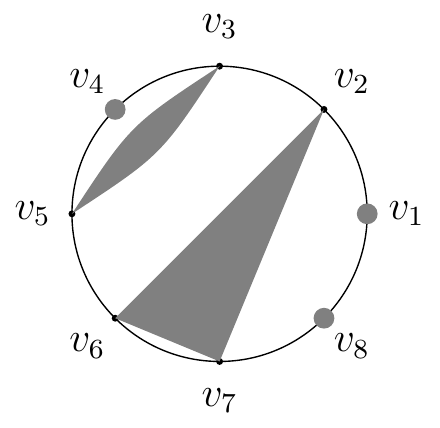}}
      \raisebox{0pt}{\includegraphics{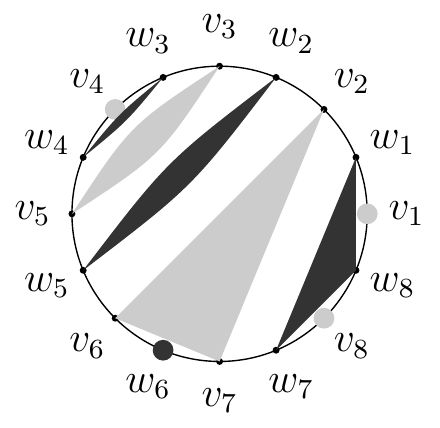}}

      \caption{The partition
      \(
        p =
        \{\,
          \{\,1\,\},
          \{\,2,6,7\,\},
          \{\,3,5\,\},
          \{\,4\,\},
          \{\,8\,\}
        \,\}
      \)
      and its Kreweras complement
      \(
        p^\mathrm{c}
        =
        \{\,
          \{\,1,7,8\,\},
          \{\,2,5\,\},
          \{\,3,4\,\},
          \{\,6\,\}
        \,\}
      \).\label{a4-c3-c13-fig-complement}}
    \end{center}
  \end{figure}    
  The set $\operatorname{NC}(m)$ of all non-crossing partitions of
  $[m]$ is partially ordered with respect
  to refinement. It is thus a subposet of the set of all partitions
  of $[m]$. It turns out that $\operatorname{NC}(m)$
  is also a lattice. This is clear from Remark~\ref{a4-c3-c13-half-a-lattice-is-a-lattice}
  since greatest lower bounds are inherited from the partition lattice
  and upper bounds exist trivially since the trivial partition with a single block
  is noncrossing.

  However, the noncrossing partition lattice is not a sublattice of the whole
  partition lattice: the join operation in both structures differ,
  i.e., the finest partition coarser than some given non-crossing partitions
  does not need to be non-crossing; see Remark~\ref{a4-c3-c13-not-a-sublattice} for a
  counterexample.
  
  The complement map
  \begin{align*}
    \operatorname{NC}(m) & \longrightarrow \operatorname{NC}(m) \\
    p & \longmapsto p^\mathrm{c}
  \end{align*}
  is an anti-automorphism of the lattice $\operatorname{NC}(m)$: it
  reverses the refinement relation and interchanges the roles of meet
  and join.  It is, however, not an involution. In the picture, taking
  the Kreweras complement twice seems to get you back to the original
  partition. This is true; however, the indexing of the vertices
  shifts by one. Thus, the square of the Kreweras complement is given
  by cyclically rotating the element of the underlying set
  $\{\,1,\ldots,m\,\}$.

  The \emph{bottom} (finest) element $\bot$ of $\operatorname{NC}(m)$
  is the partition
  with $m$ blocks, each of size one. The \emph{top} (coarsest) element
  $\top$
  of $\operatorname{NC}(m)$ is the partition with a single block. For each
  non-crossing partition $p$, we define its \emph{rank}
  $\operatorname{rk}(p)$
  in terms of its number of blocks:
  \[
    \operatorname{rk}(p) := m - \# \{\,\text{blocks of\ }p\,\}
  \]
  For any non-crossing partition $p$, all maximal chains from
  the bottom element $\bot$ to $p$ have the same length, which
  coincides with the rank $\operatorname{rk}(p)$.
  Let us summarise the properties and non-properties of the poset of
  non-crossing partitions:
  \begin{fact}
    The set $\operatorname{NC}(m)$ of non-crossing partitions of an
    $m$-element is partially ordered by refinement. This poset is a
    lattice and self-dual with respect to the Kreweras~complement,
    i.e.,
    \begin{align*}
      (p\wedge q)^\mathrm{c}
      &= p^\mathrm{c}\vee q^\mathrm{c}\\
      (p\vee q)^\mathrm{c}
      &= p^\mathrm{c}\wedge q^\mathrm{c}\\
    \end{align*}
    for any two $p,q\in\operatorname{NC}(m)$.

    The automorphism $p\mapsto(p^\mathrm{c})^\mathrm{c}$
    has order $m$.
    
    All maximal chains from bottom to top have length $m-1$.
    For any non-crossing partition $p$, there is a maximal
    chain from bottom to top going through $p$. The non-crossing
    partition lattice is graded and one has
    \[
      m-1 = \operatorname{rk}(p)+\operatorname{rk}(p^\mathrm{c})
    \]
    for any $p$.
  \end{fact}
  \begin{remark}\label{a4-c3-c13-not-a-sublattice}
    For $m\geqslant 4$, the non-crossing partition lattice
    $\operatorname{NC}(m)$ is not a sub-lattice of the partition
    lattice: the join operations do not coincide.  A counterexample
    for $m=4$ is \( p=\{\, \{\,1,3\,\},\{\,2\,\},\{\,4\,\} \,\} \)
    and \( q=\{\, \{\,1\,\},\{\,2,4\,\},\{\,3\,\} \,\} \).
    The join of these partitions in the partition lattice is
    $\{\,\{\,1,3\,\},\{\,2,4\,\}\,\}$ whereas the join in
    $\operatorname{NC}(4)$ is the top element.  These two partitions
    also show that the non-crossing partition lattice
    $\operatorname{NC}(m)$ is not \emph{semi-modular}, i.e., the
    following inequality does not hold for all partitions $p$ and $q$,
    \[
      \operatorname{rk}(p)+\operatorname{rk}(q)
      \geqslant
      \operatorname{rk}(p\vee q)
      +
      \operatorname{rk}(p\wedge q)
      .
    \]
  \end{remark}

  Enumerative properties of the noncrosing partitition lattice are
  well understood. Kreweras counted the number of non-crossing
  partitions.
  \begin{fact}[{see~\cite[Cor.~4.2]{a4-c3-c13-Kreweras}}]
    For any $m$, we have
    \[
      \left|\, 
        \operatorname{NC}(m)
      \, \right|
      =
      C_{m}
    \]
    where
    \(
      C_{m}
      = \frac{1}{m+1}
        {\binom {2m} m}
      = \frac{(2m)!}{m!(m+1)!}
    \)
    is the $m^{\text{th}}$ \emph{Catalan number}\index{Catalan number}.
  \end{fact}
  Kreweras also determined the M{\"o}bius function for the lattice of
  non-crossing partitions. Recall that, for a finite poset $P$, the
  \emph{M{\"o}bius function}\index{poset!M{\"o}bius function}
  \[
    \mu :
    \{\,\,
      (u,v)
      \in
      P \times P
    \,\,|\,\,
      u \leq v
    \,\,\}
    \longrightarrow \ZZ
  \]
  is defined by the following recursion:
  \begin{align*}
    \mu(u,u) &= 1, \\
    \mu(u,v)
    & = - \sum_{u\leq w<v}
          \mu(u,w) .
  \end{align*}
  Note that the value $\mu(u,v)$ is
  completely determined by the isomorphism type (as a poset) of the
  interval
  \(
    [u,v]
    :=
    \{\,\, w \in P \,\,|\,\, u\leq w\leq v \,\,\}
  \).
  \begin{fact}[{see~\cite[Thm.~6]{a4-c3-c13-Kreweras} 
      or \cite[Cor.~3.2]{a4-c3-c13-BlassSagan}}]
    For the non-crossing partition poset $\operatorname{NC}(m)$, the
    M{\"o}bius function satisfies
    \begin{equation}\label{a4-c3-c13-moebius}
      \mu(\bot,\top)
      =
      (-1)^{m-1}C_{m-1}
      = (-1)^{m-1} \frac{(2m-2)!}{(m-1)!m!}
    \end{equation}
  \end{fact}
  Let $p$ be a non-crossing partition, and consider a non-crossing
  partition $q\preceq p$. Let $B$ be a block
  of $p$. The blocks of $q$ contained in $B$
  may be thought of as a non-crossing partition of $B$. Thus, we have
  the following:
  \begin{observation}\label{a4-c3-c13-intervall-structure}
    Let $p\in\operatorname{NC}(m)$ be a non-crossing partition,
    and let \linebreak $B_{1},\ldots,B_{k}$ be its blocks.
    Then the \emph{order ideal}
    \(
      p_{\preceq}
      :=
      \{\,\,
        q\in\operatorname{NC}(m)
      \,\,|\,\,
        q\preceq p
      \,\,\}
    \)
    is isomorphic as a poset to the cartesian product
    \(
      \operatorname{NC}(B_{1})
      \times \cdots \times
      \operatorname{NC}(B_{k})
    \).

    Let $B'_{1},\ldots,B'_{m-k+1}$ be the blocks of the Kreweras
    complement $p^\mathrm{c}$. Since the complement is an
    antiautomorphism of the non-crossing partition lattice, the
    \emph{filter}
    \( p_{\succeq} := \{\,\, q\in\operatorname{NC}(m) \,\,|\,\,
    q\succeq p \,\,\} \)
    is isomorphic as a poset to the cartesian product
    \( \operatorname{NC}(B'_{1}) \times \cdots \times
    \operatorname{NC}(B'_{m-k+1}) \).

    For non-crossing partitions $p\preceq q$, the
    interval $[p,q]$ is the filter
    for $p$ within the order ideal of $q$. Hence,
    by combining the previous isomorphisms, we see that
    \(
      [p,q]
    \)
    is isomorphic to the product
    \(
      \prod_{B} \operatorname{NC}(B) 
    \)
    where $B$ ranges over the blocks of the ``blockwise Kreweras
    complement'' of $p$ in $q$.
  \end{observation}
  \begin{figure}[h]
    \begin{center}
      \raisebox{0pt}{\includegraphics{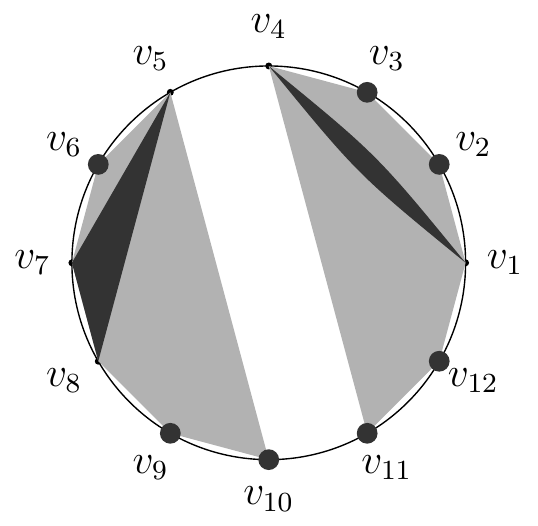}}
      \raisebox{0pt}{\includegraphics{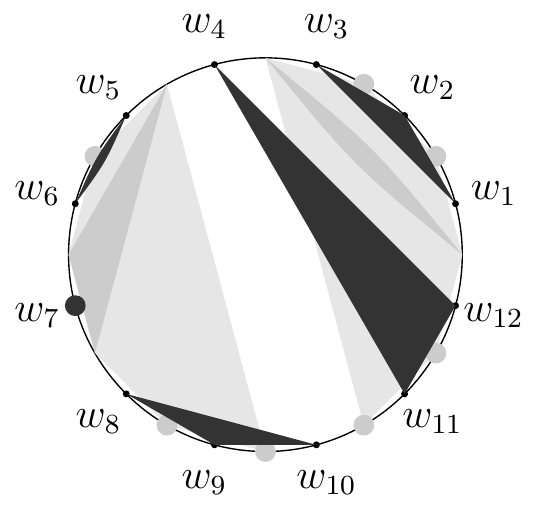}}

      \caption{Two nested partitions $p\preceq q$ and their blockwise
        complement. For the dual vertices $w_{4}$ and $w_{10}$,
        different conventions are possible to determine which dual
        vertex is to be used with which block of
        $q$.\label{a4-c3-c13-fig-blockwise-complement}}
    \end{center}
  \end{figure}
  Since the M{\"o}bius function is multiplicative with respect to
  cartesian products of posets,
  Observation~\ref{a4-c3-c13-intervall-structure} allows one to derive
  the values of $\mu(p,q)$ in terms of the blockwise complement of $p$
  in $q$ from Kreweras' formula~(\ref{a4-c3-c13-moebius}).

  \begin{remark}\label{a4-c3-c13-reversing-order}
    To every poset $(P,\leq)$, one associates the \emph{order complex}
    \index{poset!order complex}.  This is the simplicial complex
    $\Delta(P,\leq)$ whose vertices are the elements of $P$ and whose
    simplices are \emph{chains} in $P$, i.e., non-empty subsets of $P$
    on which $\leq$ is a total order. By a theorem of
    P.~Hall\index{theorem!Hall's theorem}, one can interpret the
    M{\"o}bius function as the Euler characteristic\index{Euler characteristic} 
of order complexes
    \cite[Prop.~3.8.6]{a4-c3-c13-Stanley},
    \[
      \mu(u,v)
      =
      \chi( \Delta( (u,v) ) ),
      \qquad\text{for\ }u<v.
    \]
    Here
    \(
      (u,v)
      :=
      \{\,\, w \in P \,\,|\,\,
        u < w < v
      \,\,\}
    \)
    is the \emph{open interval} from $u$ to $v$.
    
    A significant implication is that the M{\"o}bius function is
    invariant with respect to reversing the order relation: let
    $\mu_{\leq}$ be the M{\"o}bius function of $(P,\leq)$ and let
    $\mu_{\geq}$ be the M{\"o}bius function of the reversed poset
    $(P,\geq)$; then, we have
    \[
      \mu_{\leq}(u,v)
      =
      \mu_{\geq}(v,u).
    \]
  \end{remark}

  {\section{Non-crossing partitions in free probability}} 

  Classical probability spaces $(\Omega, \cF, \mathbf{P})$ can be
  reformulated using the commutative $C^*$-algebra
  $\cA= L^{\infty}(\Omega, \cF, \mathbf{P})$ as follows. Real valued
  (bounded) random variables correspond to elements of $\cA$ and their
  expectations are given by evaluation of the linear functional
  $\varphi(a):= \int_{\Omega} a d \mathbf{P} $. The 'distribution' of
  a random variable $a$ is the induced distribution
  $\mu_a(A):= \mathbf{P}(a^{-1}(A))$ and its $k$th moment is given by
  $\varphi(a^k) =\int_{\Omega} a^k d\mathbf{P}= \int_\RR
  x^k\,\mu_a(dx) = \int_\RR x
  \,\mu_{a^k}(dx)$. 

  This construction admits the following non commutative extension.
  Denote by $(M_d(\CC), \tr) $ the space of $d\times d$ complex
  matrices, together with the normalised trace and the usual matrix
  conjugation. Consider now the algebra of \emph{random matrices}
  $\cA := M_d(L^{\infty}(\Omega, \cF, \mathbf{P}))$ together with the
  linear functional $\varphi(a) := \int_{\Omega}\tr(a)
  d\mathbf{P}$. \index{random element}

  This represents a genuine \emph{non-commutative $C^*$-probability
    space} $(\cA, \varphi )$, which is a unital $C^*$-algebra over
  $\CC$ 
  together with a unital and tracial positive linear functional
  $\varphi : \cA \to \CC$, that is
$$  
\varphi(1) = 1, \qquad \varphi(a^*a)\ge 0, \quad  \varphi(ab) = \varphi(ba), \qquad \text{for all } a,b \in \cA.
$$ 
Furthermore, we shall assume that $\varphi$ is faithful, that is
$\varphi(a^*a)=0$ is equivalent to $a=0$.  See the survey
\cite{a4-c3-c13-Voiculescu:2000}.
 
Many constructions in non commutative probability are parallel to
those in classical probability, and this is also reflected in the
notation: If $a$ is a self-adjoint element in $\cA$, i.e. $a^*=a$, the
value $\varphi(a)$ is sometimes called the \emph{expectation} of $a$,
the values $\varphi(a^k)$, $k \in \NN$, are called the \emph{moments}
of $a$, and the compactly supported probability measure $\mu_a$ on
$\RR$ with $\int x^k \mu_a(dx) = \varphi(a^k)$, $k \in \NN$, is also
called the \emph{distribution} of $a$ which always exists for
self-adjoint elements in a $C^*$-probability space.  If the measure
$\mu_a$ admits a density $f_a$, the latter is also called the
\emph{density} of $a$.  Similarly, given two self-adjoint elements $a$
and $b$ in $\cA$, the \emph{joint moments} of $a$ and $b$ are given by
the values $\varphi(w)$, $w$ being a ``word'' in $a$ and $b$.
\\
Recall that a compactly supported Borel measure $\mu$ on $\RR$ (and
more generally any $\mu$ with $\int e^{zx}\mu(dx)$ locally analytic
around $z=0$) is \emph{uniquely} characterised by its moments
$\int x^k \mu(dx)$ since then the Fourier transform of $\mu$ is a
convergent power series with coefficients given by the moment
sequence. \index{distribution}

In order to define a corresponding notion of independence for
self-adjoint elements (like that for random variables in classical
probability theory), recall that two random variables
$a,b\in L^{\infty}(\Omega, \cF, \mathbf{P})$ endowed with expectation
$\varphi$ as above are \emph{independent}, if
$\varphi(a^k b^l) = \varphi(a^k) \varphi(b^l)$ or equivalently
\begin{equation} 
\varphi\Big((a^k-\varphi(a^k))(b^l-\varphi(b^l))\Big)=0
\label{a4-c3-c13-eq:indep}
\end{equation} 
for all $k,l\in \mathbf{N}_{0}$.

Let 
$\cA_1$ and $\cA_2$ denote unital sub-algebras in $\cA$, for instance
generated by elements $a$ and $b$ respectively. They are called `free'
if the expectations of all products with factors alternating between
elements from $\cA_1$ and $\cA_2$ vanish whenever the expectations of
all factors vanish. \index{freeness} Hence the elements $a,b \in \cA$
are called \emph{free} if
\begin{multline} \label{a4-c3-c13-eq:free}
\varphi\big(
(a^{j_1}-\varphi(a^{j_1}))
(b^{k_1}-\varphi(b^{k_1}))\cdots
(a^{j_m}-\varphi(a^{j_m}))
(b^{k_m}-\varphi(b^{k_m})) \big) = 0 
\end{multline}
for all $m \in \NN$ and all $j_1,\hdots,j_m,k_1,\hdots,k_m \in
\NN$.
Hence for $m=1$ this rule for the evaluation of joint moments
coincides with the classical rule $\varphi(ab)=\varphi(a)\varphi(b)$
but is apparently different for $m>1$.  The rules
\eqref{a4-c3-c13-eq:free} as well as \eqref{a4-c3-c13-eq:indep} allow
to reduce by induction the evaluation of joint moments
$\varphi(a^{j_1} b^{k_1} \cdots a^{j_m} b^{k_m})$ of these free or
independent elements to the moments $\varphi(a^j)$ and $\varphi(b^k)$,
which determine the marginal distribution of $a$ resp. $b$.  Thus
freeness may be regarded as a (non-commutative) analogue of the notion
of independence in classical probability theory, allowing the
development of a \emph{free probability theory}\index{free probability}\index{independence}.  
In particular
\eqref{a4-c3-c13-eq:free} allows to to compute the expectation of
$\varphi((a+b)^n)$ for any $n\in \NN$, $a\in \cA_1$ and $b \in \cA_2$,
thus determining the distribution in the sense described above of the
`free' sum of $a$ and $b$ via the moments of $a$ and $b$ only. Hence,
this assigns to compactly supported measures $\mu, \nu$ (with moments
given by those of $a,b$) a \emph{free additive convolution}
$\mu\boxplus \nu$, see the survey
\cite{a4-c3-c13-Voiculescu:2000}. This notion may be considered as an
asymptotic limit of a corresponding notion for sequences of random
matrices with independent entries of increasing dimension and their
limiting spectral measures, 
 \cite[Chapter 1]{sfb701}.

More generally, a set of unital sub-algebras
$\cA_j \subset \cA, \,j\in I $, indexed by a set $I$, is called free
if for any integer $k$ and
$a_j \in \cA_{i_j}, j=1, \ldots, k, i_j \in I$,
\begin{align}
  \varphi(a_1\ldots a_k)&=0 \quad \text{provided that}\quad 	\varphi(a_j) =0, \quad j=1, \ldots , k, \notag \\
                        &\quad \text{and} \quad  i_1 \ne i_2, i_2 \ne i_3,\dots,  i_{k-1} \ne i_k,\label{a4-c3-c13-eq:freek}
	\end{align}
        that is, all adjacent elements in $a_1 \ldots a_k$ belong to
        different sub-algebras $A_{j_i}$.  This notion has similar
        properties as classical independence. For instance,
        polynomials $P(a_j)$ of free self-adjoint elements $a_j$
        (generating a sub-algebra) are free again.

        The density $\psi(x)=\frac 1{\sqrt{2\pi}} \exp(-x^2/2)$
        defines the standard Gaussian distribution.  Hence, the
        \emph{classical central limit theorem (CLT)} may be stated for
        independent random elements $a_i, i \in \NN$ from a
        commutative $C^*$-probability space $(\cA, \varphi)$ with
        \emph{identical} distribution such that
        $\varphi(a_i)=0 ,\, \varphi(a_i^2)=1$ (such variables are
        called \emph{standardised}).
        \begin{theorem}[Commutative $C^*$-version of
          CLT]\index{theorem!classical central limit theorem}
	
          The moments of the normalised sum
          $S_N:=\frac{a_1+\ldots +a_N} {\sqrt{N}}$ satisfy
\begin{equation}
\lim_{N \to \infty} \varphi(S_N^k) = \int x^k \psi(x) \, dx, \quad k \in \NN.
\end{equation} 
\end{theorem}

Consider \emph{free} random elements $a_i$ from a (non-commutative)
$C^*$-probability space $(\cA, \varphi)$ , standardized via
$\varphi(a_i)=0 , \varphi(a_i^2)=1$ with \emph{identical distribution},
that is $\varphi(a_j^l)$ depends on $l$ only.  In order to describe a
corresponding \emph{free} `central limit theorem' for this setup we
have to determine the asymptotic behaviour of moments of type
$\varphi(a_{i_1}\ldots a_{i_k})$ subject to the assumption of freeness
\eqref{a4-c3-c13-eq:freek}.
 
Note that by freeness 
all mixed moments vanish provided an element $a_j$ occurs only once in
the product vanish.  (Note that this holds as well for mixed moments
of independent random variables). Thus, we only need to consider mixed
moments with factors occurring at least twice.  For a product
$a_{i_1}\cdots a_{i_k}$ of $k$ factors, such that $s$ of them, say
$b_1, \ldots, b_s$, are different, let ${ p}=\{B_1,\ldots,B_s\}$
denote the corresponding partition of the set $\{1,\ldots, k\}$ into
$|{ p}|:=s$ nonempty blocks $B_j$ of the positions of $b_j$ in
$1\le j \le k$.
  
One can show by induction that all mixed moments of free or
independent elements $\varphi(a_{i_1}a_{i_2} \cdots a_{i_k})$ where
$1\le i_j \le N$, can be computed via \eqref{a4-c3-c13-eq:freek} resp.
\eqref{a4-c3-c13-eq:indep} as above also for $s \ge 2$ in terms of
moments $c_l=\varphi(b_j^l)$ for $j=1,\ldots, s$ which depend on $l$
only by the assumption of identical distribution. Thus these mixed
moments depend \emph{on the partition scheme of} $i_1, \dots, i_k$, say
${{ p}}$, only and will be denoted by $m_{{ p}}$.  The number of such
mixed moments in $a_1, \ldots, a_N$ corresponding to a given partition
scheme depends on $|{ p}|$ only and is given by
$A_{N,{ p}}=N(N-1)\cdots (N-|{ p}|+1)$ . Thus
\begin{equation}\label{a4-c3-c13-eq:moment}
\varphi(S_N^k) = \sum_{{ p}} 
 m_{{ p}} A_{N,{ p}}N^{-k/2}.
 \end{equation}
 For a partition ${ p}$ we have $A_{N,{ p}} < N^{|{ p}|}$.  If all
 parts of ${ p}$ satisfy $|B_j|\ge 2 $ and one block is of size at
 least three, the corresponding contribution in
 \eqref{a4-c3-c13-eq:moment} is of order
 $|m_{{ p}}| A_{N,{ p}}N^{-k/2}\le |m_{{ p}}| N^{-1/2}$, that is all
 these terms are asymptotically negligible as $N$ tends to infinity.
  
 Hence, computing the asymptotic limit of $\varphi(S_N^k)$ reduces to
 considering all mixed moments of $k$ factors with each random element
 occurring precisely \emph{twice}, a consequence being that
 $\lim_{N\to\infty}\varphi(S_N^k) = 0$ for $k$ odd.

 Recall that $ \operatorname{NC}(n)$ denoted the lattice of all
 non-crossing partitions on the set $[n]=\{
 1,\ldots,n\}$.
 Furthermore, let $ \operatorname{NC}_{2}(2k)$ denote the subset of
 non-crossing partitions with blocks of size 2 only, called
 'non-crossing pair partitions' on a set of $2k$ elements.

 Now consider as an example three free standardised variables
 $a,b, c$.  Then the product $abc^2ab$ corresponds to a pair partition
 with a \emph{crossing}, that is ${
   p}=\{\{1,5\},\{3,4\},\{2,6\}\}$. Hence
 $\varphi(abc^2ab)=\varphi(abab) \varphi(c^2)= 0$ by freeness, that is
 \eqref{a4-c3-c13-eq:free}. Otherwise for a non-crossing pair
 partition like $ca^2b^2c$ we have
 $\varphi(ca^2b^2c)=\varphi(cb^2c)\varphi(a^2)=\varphi(cc)\varphi(b^2)=1$. These
 simple observations can be generalised by induction in the following
 Lemma to determine the values of joint moments
 $m_{p}= \varphi(a_{i_1}a_{i_2} \cdots a_{i_k})$ for pair partitions
 $p$ of free
 variables. 
  \begin{lemma}
    For \emph{any} pair partition ${ p}$,
    $$m_{p}=\begin{cases}0 &   \text{if $p$ has a crossing }\\
    1  & \text{if $p$ is non-crossing.} \end{cases}$$
  \end{lemma}

  Thus, we conclude from \eqref{a4-c3-c13-eq:moment} and the previous
  results that
  $$ 
\lim_{N\to \infty}\varphi(S_N^{2k})=\lim_{N\to \infty} \sum_{{ p} \in  \operatorname{NC}_{2}(2k)} \frac{A_{N, { p}}}{N^{k/2}}= | \operatorname{NC}_{2}(2k)|.
$$
  Furthermore, one shows that
   \begin{equation}
     C_k:=| \operatorname{NC}_{2}(2k)|=| \operatorname{NC}(k)|,		
   \end{equation}
   where $C_k=\frac{1}{k+1} \binom{2k}{k}$ is the $k$th Catalan
   number. Among its numerous interpretations, it represents as well
   the $2k$\,th moment of a compactly supported measure with density
   $w(x):=\frac 1{2 \pi} (4-x^2)^{1/2}, |x| \le 2$.  This is the
   so-called \emph{Wigner} measure or semi-circular distribution. See
   \cite[Rem.~9.5]{a4-c3-c13-Nica-Speicher:2006}. \index{Wigner measure}

   Now the \emph{free central limit theorem} for a sequence of free
   variables $a_j, j \in \NN$, which are standardised via
   $\varphi(a_j)=0, \, \varphi(a_j^2)=1$, and
   $S_N:=\frac{a_1+\ldots +a_N} {\sqrt{N}}$ may be stated as follows.
   \begin{theorem}[Free Central Limit Theorem]\index{theorem!free central limit theorem}
     $S_N$ converges in distribution to $w$ which serves as the Gaussian
     distribution in free probability, i.e.
     \begin{equation}\label{a4-c3-c13-eq:momentfree}
       \lim_{N \to \infty} \varphi(S_N^k) = \int x^k w(x) dx, \quad k \in \NN.
     \end{equation}
   \end{theorem}
   This means e.g. that the rescaled sum $(a_1+a_2)/\sqrt2$ of two
   free elements $a_1,a_2$ of a non-commutative probability space
   $(\cA, \varphi)$ which both have density $w(x)$ again has a Wigner
   distribution.  In free probability an element $s$ of
   $(\cA, \varphi) $ with density $w(x)$ is called \emph{semi-circular}
   and its moments are given by
   \begin{equation}\label{a4-c3-c13-eq:smoment}
     \varphi(s^n) = \begin{cases} \frac{1}{k+1} \binom{2k}{k}, &\mbox{if}\quad n=2k,\\
       0, & \mbox{if} \quad n  \text{ odd}. \end{cases}
   \end{equation}
   
   Recall that $a\in (\cA, \varphi)$ is called \emph{positive} if there
   exists an $c \in (\cA, \varphi)$ with $a=c^* c$ . Thus $a$ is
   self-adjoint.  Define the \emph{free multiplicative convolution} of
   two compactly supported measures $\mu_a,\mu_b$, of \emph{positive}
   free elements $a, b \in (\cA, \varphi)$, say
   $\mu_a\boxtimes \mu_b$, as follows by specifying its moments.
   \index{free multiplicative convolution}
   Since in a $C^*$-probability space $\cA$ positive square roots
   $ a^{1/2}$ resp. $b^{1/2}$ of $a$ resp. $b$ as well as the positive
   element $p_{a,b}:= a^{1/2}ba^{1/2}$ are again in $\cA$, we may
   define $\mu_a\boxtimes \mu_b$ by:
   \begin{equation}
   	  \int x^k d \mu_a\boxtimes \mu_b (x) :=  \varphi(p_{a,b}^k), \qquad k \in \NN.
   	\end{equation}
  Since $\varphi(p_{a,b}^k)= \varphi(p_{b,a}^k), \, k \in \NN$, 
  because $\varphi$ is tracial, i.e. $\varphi(ba)=\varphi(ab)$, we
  conclude that the free convolution $\boxtimes$ is commutative.  By
  the same tracial property and the relation of freeness, we show that
  $\varphi(p_{a,b}^k)= \varphi((ab)^k)$ and this implies the
  associativity of $\boxtimes$.  Moreover it follows from this
  representation 
  that the multiplicative convolution measure $\mu_a\boxtimes \mu_b$
  is uniquely determined by the distributions of $\mu_a$ and
  $\mu_b$. \index{free additive convolution}
     
  In order to effectively compute both additive and multiplicative
  convolution of measures, one needs more properties of the lattice of
  partitions of $1, \ldots, n$ into blocks and the subset of
  non-crossing partitions together with the notion of multi-linear
  cumulant functionals.  As above let $B_j, j=1,\ldots s$ denote the
  blocks of a partition ${ p}\in \operatorname{NC}(n)$ of
  $1, \ldots, n$.

  For ${p}\in\operatorname{NC}(n)$, the \emph{free mixed cumulants} are
  multi-linear functionals $\kappa_{{p}}: \cA^n \to
  \CC$ 
  defined in terms of a moment decomposition using the M{\"o}bius
  function $ \mu({ q}, { p})$
  of the lattice of non-crossing partitions  $ \operatorname{NC}(n)$.
  We define the general mixed cumulant functionals $\kappa_{{ p}}$ as follows:
    \begin{align} \label{a4-c3-c13-eq:cumdef}
      \kappa_{{ p}}[a_1, \ldots ,a_n] &
      =
      \sum_{ q \in  \operatorname{NC}(n), { p}\preceq{ q}}
        \varphi_{{  q}}[a_1, \ldots, a_n] \,
         \mu({ p}, { q}), \quad \text{where} \\
      \varphi_{ q}[a_1, \ldots, a_n] &
      :=
      \varphi\left(\prod_{k\in B_{1}} a_k\right)
      \cdots  \varphi\left(\prod_{k\in B_{s}} a_k\right),\notag
    \end{align}
    and the products $\prod_{k\in B_{j}} a_k$ repeat the order of
    indices within the block $B_j$.  Note that by Hall's theorem, the
    coefficient $ \mu( p, q)$ can also be written as
    $ \mu_{\succeq}( q, p)$ using the relation of reversed refinement
    (see~Remark~\ref{a4-c3-c13-reversing-order}).
  
    Then one shows, see
    \cite[Prop.~11.4]{a4-c3-c13-Nica-Speicher:2006}, that
\begin{equation}\label{a4-c3-c13-eq:momcum}
\varphi(a_1 \cdots a_n) = \sum_{{ p}\in  \operatorname{NC}(n)} \kappa_{{ p}}[a_1,\ldots, a_n].
\end{equation}

In the special case $p={1_n}$ we write $\kappa_n$ instead of
$\kappa_{1_n}$. The following lemma is proved by induction on $n$.
\begin{lemma}[{\cite[Thm~11.20]{a4-c3-c13-Nica-Speicher:2006}}]\label{a4-c3-c13-freechar}
  The elements $a_1, \ldots, a_n \in \cA$ are free if and only if all
  mixed cumulants satisfy
$$ \kappa_n[a_{j_1}, \ldots ,a_{j_k}] =0,$$
 whenever $ a_{j_1}, \ldots a_{j_k}$,
$1\le j_l\le n, 1\le k\le n$ 
contains at  least two different elements. 
\end{lemma}

In contrast to \eqref{a4-c3-c13-eq:freek}, this characterisation of
freeness holds even if the $\varphi(a_j)$ are non-zero.

For a partition ${ p} \in \operatorname{NC}(n)$, recall that
${ p}^\mathrm{c}$ denotes its \emph{Kreweras}~complement in
$ \operatorname{NC}(n)$.  Then, one shows that for free elements $a,b$
the following recursion involving the Kreweras complement holds:
	\begin{equation}\label{a4-c3-c13-eq:recursion}
	\kappa_n[ab, \ldots, ab]= \sum_{{ p} \in  \operatorname{NC}(n)} \kappa_{{ p}}[a,\ldots,a] \kappa_{{ p}^\mathrm{c}}[b,\ldots, b].
	\end{equation}
        See \cite[Rem.~14.5]{a4-c3-c13-Nica-Speicher:2006}.  This
        entails that the cumulants of $ab$ and thus by
        \eqref{a4-c3-c13-eq:momcum} the moments of $ab$ are indeed
        determined by multi-linear functionals of $a$ and $b$ alone
        which again by virtue of \eqref{a4-c3-c13-eq:cumdef} are
        determined by the moments of $a$ together with the moments of
        $b$. 
	
	The recursive equation \eqref{a4-c3-c13-eq:recursion} and the
        definition \eqref{a4-c3-c13-eq:cumdef} of cumulants may be
        conveniently encoded as algebraic relations between the
        following formal generating series. For $a \in \cA$ let
        $M_a(z)= \sum_{n=1}^\infty \varphi(a^n)z^n$ denote the moment
        generating series and with $\kappa_n(a):=\kappa_n[a,\ldots,a]$
        let $R_a(z):= \sum_{n=1}^{\infty} \kappa_n(a)z^n$ and
        $\cR_a(z):= z^{-1}R_a(z)$ denote cumulant generating
        series. In particular, for free self-adjoint $a,b \in \cA$ we
        get by binomial expansion of $\kappa_n(a+b)$ and Lemma
        \ref{a4-c3-c13-freechar} that
        $\kappa_n(a+b)=\kappa_n(a)+\kappa_n(b)$ and furthermore, as
        shown in \cite[Lect.~12]{a4-c3-c13-Nica-Speicher:2006},
 \begin{lemma}
   One has the following identities:
	\begin{align}
	R_{a+b}(z) &= R_a(z)+R_b(z),\label{a4-c3-c13-eq:rhom} \\
	R_a(zM_a(z)+z) &= M_a(z),\quad G_a\Big(\frac{1+R_a(z)}{z}\Big) =z,\label{a4-c3-c13-eq:rdef}
		\end{align}
                where
                $$G_a(z):=\frac{1}{z}+ \sum_{n=1}^\infty
                \frac{\varphi(a^n)}{z^{n+1}}=\frac1 z \Big(1+M_a(\frac
                1 z)\Big),$$
                can be identified with the Cauchy transform of the
                corresponding spectral measure $\mu_a$, that
                is $$G_a(z)=\int_\RR\frac{d\mu_a(t)}{z-t}.$$
\end{lemma}
 Hence the so-called
   R-transform $\cR$ of a spectral measure 
   $\mu_a$, introduced by Voiculescu in \cite{a4-c3-c13-Voiculescu:2000}, is determined analytically
   by the inverse function of the Cauchy transform 
   of $\mu_a$ on the complex plane which is the 
   starting point of the complex analytic theory
   of the asymptotic approximations of free additive convolution as developed 
   in \cite{a4-c3-c13-CG11,a4-c3-c13-CG08a,a4-c3-c13-CG08b, A4-CG13}.
   Assuming that $\kappa_1=m_1  \ne 0$, $R_{\mu_a}(z):=R_a(z)$
   admits a formal inverse power series $R_a^{(-1)}(z)$. This 
   may be defined via the inverse function of
   the Cauchy transform of $\mu_a$, which is well defined in a certain  region in $\CC$.\index{R-transform of Voiculescu@$R$-transform of Voiculescu}
   
   The  so-called \emph{$S$-transform}
   \begin{equation}\label{a4-c3-c13-eq:rsrel}
   S_a(z):=\frac1 zR_a^{(-1)}(z) =\frac{1+z}{z} M_a^{(-1)}(z),
   \end{equation}
   of Voiculescu is a multiplicative homomorphism for free
   multiplicative convolution. \index{S-transform of Voiculescu@$S$-transform of Voiculescu}
   That is, see \cite[Lect.~18]{a4-c3-c13-Nica-Speicher:2006}, one has
   the following result.
   \begin{lemma}
   For two free self-adjoint positive elements $a,b\in \cA$, one has
   \begin{equation} 
   S_{ab}(z)= S_a(z) S_b(z)
   \end{equation}
   \end{lemma}
   Since $S_a$ is determined by the spectral measure of $a$, this
   means with $S_{\mu_a}:=S_a$ for measures $\mu =\mu_a$, $\nu=\mu_b$
   we have $ S_{\mu\boxtimes \nu}(z)= S_{\mu}(z) S_{\nu}(z)$, which
   uniquely determines the multiplicative 
   free convolution $\mu\boxtimes \nu$ in terms of the measures $\mu$
   and $\nu$ on the positive reals via the characterising property of
   the $S$-transform.

   Note that by \eqref{a4-c3-c13-eq:smoment}, Let $s$ be a
   semi-circular element as in \eqref{a4-c3-c13-eq:smoment}. Then the
   moment generating functions of $s$ and $s^2$ are given by
   $M_{s}=f(z^2)$ and $M_{s^2}(z)=f(z)$ respectively, where
   $f(z) = (1- \sqrt{1- 4z})/(2z) -1$.  The corresponding distribution
   of $s^2$ is called Marchenko-Pastur or free Poisson law; it is
   given by the density $p(x):=\frac 1 {2\pi} \sqrt{4/x-1}$ on the
   interval $[0,4]$.  Via the inverse function
   $f^{(-1)}(z)=z(1 + z)^{-2}$ of $f$ we obtain in view of
   \eqref{a4-c3-c13-eq:rsrel},
\begin{equation}\label{a4-c3-c13-eq:smp}
S_{s^2}(z)=f^{(-1)}(z)\frac{1+z}{z}= \frac 1{1+z}
\end{equation}
and hence in view of \eqref{a4-c3-c13-eq:rsrel} again
$R_{s^2}^{(-1)}(z)=\frac{z}{1+z}$ or $R_{s^2}(z)= \frac{z}{1-z}$,
whereas from \eqref{a4-c3-c13-eq:rdef} we deduce with
$g(z):=z(1+M_a(z))$ and $g^{(-1)}(z)=\frac{z}{1+z^2}$ and hence
$R_s(z)= \frac{z} {g^{(-1)}(z)} -1=z^2$.

From here, we obtain for free variables $t_1, \ldots, t_l$ with
identical distribution given by $s^2$, the so-called Marchenko--Pastur
distribution, in view of \eqref{a4-c3-c13-eq:smp}
\begin{equation}
S_{t_1\ldots t_l}(z) = S_{t_1}(z)^l = \frac{1} {(1+z)^l},
\end{equation}
which determines the so-called free Bessel distributions, $\mu_l$ with
support in $[0,K_l]$, $K_l=(l+1)^{l+1}/l^l$. Their moments are given
by the so called Fuss--Catalan numbers, that is, if an element
$a\in \cA$ has $S$-transform $S_a(z)=\frac{1}{(1+z)^l}$ we have
\begin{equation}
  \varphi(a^k)=\frac{1}{lk+1}\binom{lk+1}{k}=:C_{k,l}, \qquad \text{for all } \,  k \ge 1.
\end{equation}
The proof is based on combinatorial properties of non crossing
partitions, see
\cite{a4-c3-c13-Banica-Belinschi-Capitaine-Collins:2011}.

\begin{proposition}
  For a sequence of $N\times N$ independent non-Hermitian random
  matrices, $G_1, \ldots G_l$, with independent Gaussian centered
  entries with variance $1/N$, let $W:=G_1\cdots G_l$. Consider the
  normalised moments of $W W^*$.  As $N\to \infty$ they converge as
  follows:
\begin{equation}\label{a4-c3-c13-eq:rmtmom}
\lim_{N\to \infty}\frac 1 N  \int_{\Omega} 	\tr(W W^*)^l \, d\mathbf{P} = \int_0^{K_l} x^k  d\mu_l  = C_{k,l} 
\end{equation}
\end{proposition}
This can be shown by induction,  using
\begin{equation}
\tr(WW^*)^k= \tr(G_1(G_2\cdots G_l G_l^* \cdots G^*_1G_1)^{l-1}G_2\cdots G_l\cdots G^*_l \cdots G^*_1),
\end{equation}
which by moving $G_1$ to the right yields
\begin{align*}
  & \tr ((G_2\cdots G_l G_l^* \cdots G^*_1G_1)^{l-1}
  G_2\cdots G_l\cdots G^*_l \cdots G^*_1G_1)
  \\
  & \qquad\qquad\qquad\qquad\qquad\qquad = \tr((G_2\cdots G_l G_l^* \cdots G^*_1G_1)^l)
  \\
  & \qquad\qquad\qquad\qquad\qquad\qquad = \tr\Bigl((G_2\cdots G_l G_l^* \cdots G_2^*)(G_1^*G_1)\Bigr)^l.
\end{align*}
Since $(G_2\cdots G_l G_l^* \cdots G_2^*)$ and $G_1^*G_1$ are
asymptotically free 
of this volume, we get by induction for the asymptotic distribution of
$\pi_l$ the recursion $\pi_l = \pi_{l-1}\boxtimes\pi_1$, where $\pi_1$
can be identified with the limiting Marchenko--Pastur distribution of
$G_1G_1^*$. For arbitrary $N\times N$ independent Wigner matrices
(which are Hermitian matrices with entries which are independent
random variables unless restricted by symmetry) the relation
\eqref{a4-c3-c13-eq:rmtmom} has been shown by combinatorial techniques
after an appropriate regularization in \cite{a4-c3-c13-AGT10}. For
more details on the asymptotic spectral distribution of products of
so-called Girko--Ginibre matrices (having independent and identically
distributed random entries) and their inverses using the free
probability calculus, see \cite{a4-c3-c13-GKT15}. Strictly speaking
one needs to extend the non-commutative $C^*$-probability spaces to
spaces of unbounded operators to include distributions with
non-compact support like those of Gaussian matrices see
e.g. \cite{a4-c3-c13-CG11}.

Remarkably, the same results hold for powers instead of products.
Since $G_1^{l-1} (G_1^{l-1})^*$ and $G_1^* G_1$
are also asymptotically free,
a similar argument as above shows that
the asymptotic distribution of $(G_1^l) (G_1^l)^*$
is also given by $\pi_l$.
Similarly as above, these results also extend
to powers of non-Gaussian random matrices.

The calculus of $S$-transforms may even be used
to describe the asymptotic spectral measure of $W W^*$
when some of the factors in $W = G_1 \cdots G_l$ are inverted,
after appropriate regularisation of the inverse matrices
\cite{a4-c3-c13-GKT15}.
For instance, for $W = G_1 G_2^{-1}$,
the limiting distribution of $W W^*$
is given by the square of a Cauchy distribution.

Moreover, the calculus of $R$-transforms
makes it possible, at least in principle,
to deal with the case where $W$
is a sum of independent products as above
\cite{a4-c3-c13-KT15}.
For instance, for $W = G_1 G_2^{-1} + G_3 G_4^{-1}$,
the limiting distribution of $W W^*$
is also given by the square of a Cauchy distribution.
This is related to the Cauchy distribution
being ``stable'' under free additive convolution.

  \section{Braid groups}\label{a4-c3-c13-braid-groups}
    Let $\mathbb{D}$ be the unit disk. The \emph{braid group}
    $\cB_{n}$ on $n$ strands can be defined
    as the fundamental group of the \emph{configuration space}
    \[
      X_{n}:=
      \{\,\,
        \{\,z_{1},\ldots,z_{n}\,\}
        \subset
        \mathbb{D}
      \,\,|\,\,
        z_{i}\neq z_{j}
        \text{\ for\ }
        i\neq j
      \,\,\}
    \]
    of unordered $n$-point-subsets in $\mathbb{D}$. One can visualize
    a path in $X_{n}$ as a collection of $n$ distinct points moving
    continuously in $\mathbb{D}$ subject only to the restriction that
    points are not allowed to collide. Since $X_{n}$ is connected, the
    braid group (up to isomorphism) does not depend on the choice of a
    base point.
    
    We find it convenient to choose as the base point a set
    \(
      S
      =
      \{\,
        v_{1},\ldots,v_{n}
      \,\}
    \)
    of $n$ points on the boundary circle $\partial\, \mathbb{D}$
    numbered in counter-clockwise order.

    Then, we regard $\operatorname{NC}(n)$ as the poset
    of non-crossing partitions of the set $S$, i.e., for any
    two distinct blocks of the partition, their convex hulls do not
    intersect.
    A non-crossing partition $p\in\operatorname{NC}(n)$ can be
    interpreted as a braid on $n$ strands as follows: for each
    block
    \(
      B=
      \{\,
        v_{\alpha_{1}},\ldots,v_{\alpha_{k}}
      \,\}
    \), consider the counter-clockwise rotation
    of the block by one step:
    \[
      \varrho_{B} :
      v_{\alpha_{1}} \mapsto
      v_{{\alpha_{2}}} \mapsto \cdots \mapsto
      v_{\alpha_{k}} \mapsto v_{\alpha_{1}}
    \]
    The product
    \[
      \sigma_{p}
      :=
      \prod_{B\,:\text{\ block of\ }p}
        \varrho_{B} 
    \]
    describes a loop in the configuration space $X_{n}$,
    which does not depend (up to homotopy relative to the basepoint) on the
    order of factors.
    We identify it with the corresponding element of the
    fundamental group $\cB_{n}$.
    \begin{figure}[h]
      \begin{center}
        \raisebox{0pt}{\includegraphics{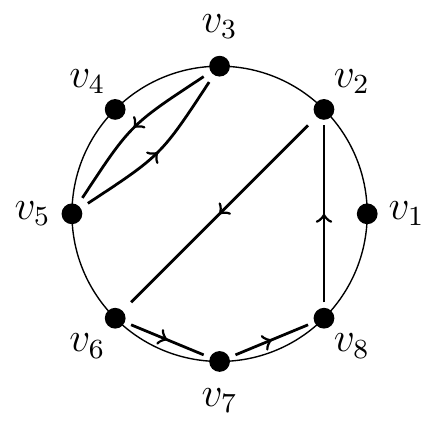}}
  
        \caption{The path (braid) that is associated to the partition     
        \(
          \{\,
            \{\,1\,\},
            \{\,2,6,7,8\,\},
            \{\,3,5\,\},
            \{\,4\,\}
          \,\}
        \).\label{a4-c3-c13-fig-braid}}
      \end{center}
    \end{figure}
    \begin{fact}\label{a4-c3-c13-artin-presentation-braid-groups}
      \index{braid group!Artin presentation of}
      The braid group $\cB_{n}$ is generated by
      the braids $\sigma_{i}$ corresponding to the
      counter-clockwise rotations
      \(
        v_{i}\mapsto v_{i+1}\mapsto v_{i}
      \)
      for $i=1,\ldots,n-1$.

      In terms of these generators, the braid group $\cB_{n}$
      admits the following presentation:
      \[
        \cB_{n}
        =
        \left\langle\,{
          \sigma_{1},\ldots,\sigma_{n-1}
        }\,\,\,\vrule\,\,
          \begin{array}{ll}
            \sigma_{i}\sigma_{j}
            =
            \sigma_{j}\sigma_{i}
            &
            \text{\ for\ }
            |i-j| \geqslant 2
            \\
            \sigma_{i}\sigma_{j}\sigma_{i}
            =
            \sigma_{j}\sigma_{i}\sigma_{j}
            &
            \text{\ for\ }
            |i-j| =1
          \end{array}
        \,\right\rangle
      \]
    \end{fact}
    There is an obvious homomorphism
    \[
      \pi:
      \cB_{n}\longrightarrow S_{n}
    \]
    from the braid group on $n$ strands to the symmetric group on $n$
    letters. A braid corresponds to a motion of the $n$ points
    $v_{1},\ldots,v_{n}$, and at the end of this motion, the dots may
    have changed positions. This way, each braid induces a
    permutation.
    \begin{fact}
      The homomorphism 
      \(
        \pi:
        \cB_{n}\longrightarrow S_{n}
      \)
      is onto. On the level of presentations, it amounts to making the
      generators $\sigma_{i}$ involutions.
      Formally: the symmetric group has the presentation
      \[
        S_{n} =
        \left\langle\,{
          s_{1},\ldots,s_{n-1}
        }\,\,\,\vrule\,\,
          \begin{array}{ll}
            s_{i}s_{j}
            =
            s_{j}s_{i}
            &
            \text{\ for\ }
            |i-j| \geqslant 2
            \\
            s_{i}s_{j}s_{i}
            =
            s_{j}s_{i}s_{j}
            &
            \text{\ for\ }
            |i-j| =1
            \\
            s_{i}=s_{i}^{-1}
            &
            \text{\ for all\ }i
          \end{array}
        \,\right\rangle
      \]
      and the homomorphism $\pi$ is sending $\sigma_{i}$
      to $s_{i}$.
    \end{fact}

    \emph{Strand diagrams}\index{braid!strand diagram}
    are another frequently used visual representation of
    braids. Recall that a braid is given by a path in configuration space, i.e.
    the simultaneous motion of $n$ points in the disk $\mathbb{D}$.
    Parametrizing time by a real number in $[0,1]$, each of those moving points
    traces out a ``strand'' in $\mathbb{D}\times[0,1]$. The diagrams we have
    used so far can be regarded as a ``top view'' onto the cylinder
    $\mathbb{D}\times[0,1]$. A strand diagram is a view from the front.
    Here, it is useful to put the initial configuration $U$ with the
    hemicircle fully visible from the front.
    Figure~\ref{a4-c3-c13-fig-strand-diagram} shows the two representations
    of the generator $\sigma_{2}$ in $\cB_{5}$. Here, the
    generator $\sigma_{i}$ corresponds to a \emph{crossing}
    of the $i^\text{th}$ and the $(i+1)^\text{th}$ strands.
    The left strand runs over the right strand. We call such a crossing
    \emph{positive}. The inverses of the generators correspond to
    \emph{negative} crossings.
    \begin{figure}[h]
      \begin{center}
        \raisebox{0pt}{\includegraphics{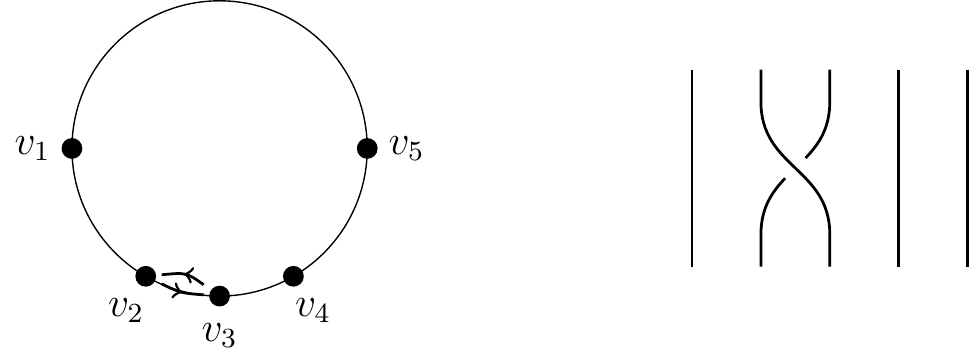}}
          
        \caption{The generator $\sigma_{2}$ in the braid group on five strands.
          On the left, the ``top view'' representation is shown whereas and on the
          right we have the ``front view'' given by a strand
          diagram.\label{a4-c3-c13-fig-strand-diagram}}
      \end{center}
    \end{figure}

    \subsection{A classifying space for the braid group}
    \label{a4-c3-c13-classifingy-space-for-braid-group}
    Tom~Brady~\cite{a4-c3-c13-Brady2001} has given a construction of a
    classifying space for braid groups that is strongly related to
    non-crossing partitions and has found some interesting applications.
    
    Recall that the \emph{Cayley graph}
    \(
      \operatorname{CG}_{\Sigma}(G)
    \)
    of a group $G$ relative to a specified generating set
    $\Sigma$ is the graph with vertex set $G$ and edges
    connecting $g$ to $gx$
    for any $g\in G$ and
    \(
      x\in\Sigma\setminus\{\,1\,\}
    \).
    Note that the requirement $x\neq 1$ rules out loops.
    Obviously, there is more structure here: the edge is oriented
    from $g$ to $gx$ and should
    be regarded as labeled by the generator $x$.
    \begin{observation}
      Since $\Sigma$ is a generating set for $G$, the Cayley
      graph $\operatorname{CG}_{\Sigma}(G)$ is connected: if we can write
      an element $g$ as a word
      \[
        g
        =
        x_{1}^{\varepsilon_{1}}
        \cdots
        x_{k}^{\varepsilon_{k}}
      \]
      in the generators and their inverses, then
      \[
        1
        \,\,-\negthinspace\negthinspace\negthinspace-\,\,
          x_{1}^{\varepsilon_{1}}
        \,\,-\negthinspace\negthinspace\negthinspace-\,\,
          x_{1}^{\varepsilon_{1}}
          x_{2}^{{\varepsilon_{2}}}
        \,\,-\negthinspace\negthinspace\negthinspace-\,\,
          x_{1}^{\varepsilon_{1}}
          x_{2}^{{\varepsilon_{2}}}
          x_{3}^{{\varepsilon_{3}}}
        \,\,-\negthinspace\negthinspace\negthinspace-\,\, \cdots \,\,-\negthinspace\negthinspace\negthinspace-\,\,
        g
      \]
      is an edge path connecting the identity element $1$ to $g$.
      Note that the exponents of the generators tell us whether to traverse
      edges with or against their orientation.\qed      
    \end{observation}
    
    There are two generating sets for the braid group (and the symmetric group)
    of particular interest to us. First, we consider the digon generators
    \(
      \sigma_{ij}
    \)
    corresponding to the counter-clockwise rotation
    \(
      v_{i}\mapsto v_{j}\mapsto v_{i}
    \).
    Let $\cB^{*}_{n}$ be the
    Birman--Ko--Lee-monoid~\cite[Section~2]{a4-c3-c13-BKL}, i.e.,
    the monoid generated by all the
    $\sigma_{ij}$. We remark that $\cB^{*}_{n}$
    is strictly larger than the submonoid of positive braids (those that can
    be drawn using positive crossings only), which is the monoid generated by the
    $\sigma_{i}$. We define a partial order on the braid group by:
    \[
      \beta\leq\beta'
      \quad:\Longleftrightarrow\quad
      \beta^{-1}\beta'\in\cB^{*}_{n}
    \]
    
    The image $s_{ij}\in S_{n}$ of
    $\sigma_{ij}$ in the symmetric group is a transposition.
    Consider the Cayley graph of the symmetric group $S_{n}$
    with respect to the generating set $T\subseteq S_{n}$
    of all transpositions. We define a partial order,
    called the \emph{absolute order}\index{absolute order},
    on $S_{n}$ as follows: For permutations
    $\xi,\psi\in S_{n}$ we declare
    $\xi\leq_{T}\psi$ if there is a geodesic
    (i.e., shortest possible) path
    in the Cayley graph connecting the identity 1 to $\psi$ and
    passing through $\xi$.

    Our largest generating set is:
    \[
      \Gamma_{n} :=
      \{\,\,{
        \sigma_{p}
      }\,\,|\,\,
        p\in\operatorname{NC}(n)
      \,\,\}
      \subseteq\cB_{n}
    \]
    which is in $1$-$1$ correspondence to the non-crossing partition lattice.
    Let $s_{p}$ denote the image of $\sigma_{p}$ in
    the symmetric group $S_{n}$. It turns out that
    the subset
    \(
      \{\,\,{ s_{p} }\,\,|\,\,p\in\operatorname{NC}(n)\,\,\}
      \subseteq
      S_{n}
    \)
    is the order ideal
    of the $n$-cycle
    \(
      1\mapsto2\mapsto\cdots\mapsto n\mapsto 1
    \)
    with respect to the partial order $\leq_{T}$ just defined, that is the subset consists of all elements in $S_{n}$ bounded above by the $n$-cycle.
    In fact,
    we have isomorphisms of various posets:
    \begin{fact}[see \cite{a4-c3-c13-Biane,a4-c3-c13-Brady2001}]\label{a4-c3-c13-set}
      Let $p,q\in\operatorname{NC}(n)$. Then the following
      are equivalent:
      \begin{enumerate}
        \item\label{a4-c3-c13-set-a}
          In $\operatorname{NC}(n)$, we have $p\preceq q$.
        \item\label{a4-c3-c13-set-b}
          In $\Gamma_{n}$, the element $\sigma_{p}$ is a
          left-divisor of $\sigma_{q}$, i.e., there exists
          $r\in\operatorname{NC}(n)$ such that
          \[
            \sigma_{q}=\sigma_{p}\sigma_{r}
          \]
        \item\label{a4-c3-c13-set-c}
          In $\Gamma_{n}$, the element $\sigma_{p}$ is a
          right-divisor of $\sigma_{q}$, i.e., there exists
          $r\in\operatorname{NC}(n)$ such that
          \[
            \sigma_{q}=\sigma_{r}\sigma_{p}
          \]
        \item\label{a4-c3-c13-set-d}
          In the braid group $\cB_{n}$, we have
          $\sigma_{p}\leq\sigma_{q}$.
        \item\label{a4-c3-c13-set-e}
          In the symmetric group $S_{n}$, we have
          $s_{p}\leq_{T}s_{q}$.
      \end{enumerate}
      Thus, on $\Gamma_{n}$ the three partial orderings given by
      left-divisibility, right-divisibility, and the partial order
      $\leq$ from $\cB_{n}$ coincide. Moreover,
      we have isomorphisms
      \[
        \operatorname{NC}(n)
        \cong
        \{\,\,{ \sigma_{p} }\,\,|\,\,p\in\operatorname{NC}(n)\,\,\}
        \cong
        \{\,\,{ s_{p} }\,\,|\,\,p\in\operatorname{NC}(n)\,\,\}
      \]
      of posets.
    \end{fact}
    \begin{figure}
      \begin{align*}
      \begin{array}{c}
      \raisebox{0pt}{\includegraphics{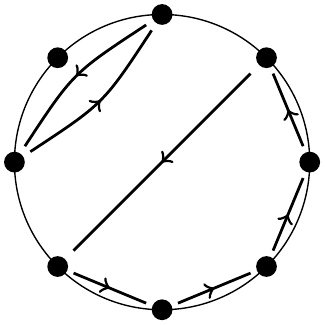}}
      \end{array}
      &=
      \begin{array}{c}
      \raisebox{0pt}{\includegraphics{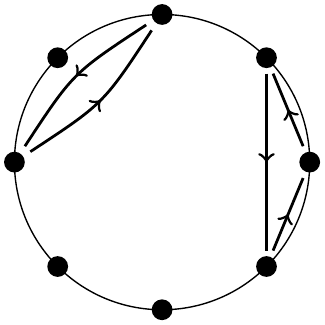}}
      \end{array}
      \circ
      \begin{array}{c}
      \raisebox{0pt}{\includegraphics{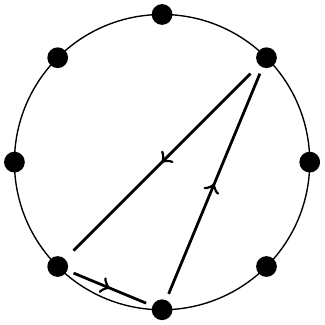}}
      \end{array}
      \\
      &=
      \begin{array}{c}
      \raisebox{0pt}{\includegraphics{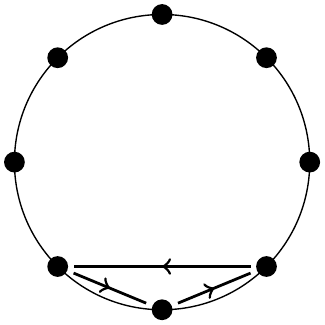}}
      \end{array}
      \circ
      \begin{array}{c}
      \raisebox{0pt}{\includegraphics{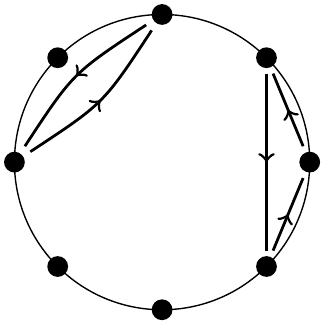}}
      \end{array}
      \end{align*}
      \caption{Left and right divisibility in $\Gamma_{8}$.\label{a4-c3-c13-division}}
    \end{figure}
    \begin{example}
      Consider the non-crossing partititions
      \[
        p :=
        \{\,
          \{\,1,2,8\,\},
          \{\,3,5\,\},
          \{\,4\,\},
          \{\,6\,\},
          \{\,7\,\}
        \,\}
        \quad\text{and}\quad
        q :=
        \{\,
          \{\,1,2,6,7,8\,\},
          \{\,3,5\,\},
          \{\,4\,\}
        \,\}
      \]
      in $\operatorname{NC}(8)$. Here, $p\leq q$ holds and we expect
      $\sigma_{p}$ to be a left- and right-divisor of $q$ within
      $\Gamma_{8}$. Figure~\ref{a4-c3-c13-division} shows the corresponding
      factorisations. One can interpret the complementary divisors as
      the blockwise Kreweras complements. In particular, the Kreweras complement
      yields factorisations of the maximal element in $\Gamma_{n}$.
    \end{example}

    The braid group $\cB_{n}$ has a particularly nice
    presentation over the generating set $\Gamma_{n}$:
    \begin{fact}[{\cite[Thm.~4.8]{a4-c3-c13-Brady2001}}]\label{a4-c3-c13-triangular-presentation}
      The valid equations
      \begin{equation}\label{a4-c3-c13-triangular-relations}
        \sigma_{1}\sigma_{2}=\sigma_{3}
        \text{\ for\ }
        \sigma_{1},\sigma_{2},\sigma_{3}\in\Gamma_{n}\setminus\{\,1\,\}
      \end{equation}
      are a defining set of triangular relations for the braid group
      \(
        \cB_{n}
      \)
      with respect to the generating set
      \(
        \Gamma_{n}\setminus\{\,1\,\}
      \).
    \end{fact}

    Let $\tilde{\Gamma}_{n}$ be the Cayley graph of the braid
    group $\cB_{n}$ with respect to the generating
    set $\Gamma_{n}\setminus\{\,1\,\}$. A \emph{clique}
    in $\tilde{\Gamma}_{n}$ is a set of vertices that are pairwise connected via
    an edge. As a directed graph, $\tilde{\Gamma}_{n}$ does not have
    oriented cycles and each clique is totally ordered by the orientation of
    edges. Thus, a clique is of the form
    \[
      \{\,
        \beta,
        \beta\sigma_{{p_{1}}},
        \beta\sigma_{{p_{2}}},
        \ldots,
        \beta\sigma_{{p_{k}}}
      \,\}
    \]
    where
    \(
      p_{1}\prec p_{2}
      \prec\cdots\prec
      p_{k}
    \)
    is an ascending chain in $\operatorname{NC}(n)$,
    and $\beta \in \cB_{n}$ is some element.
    We denote by $\tilde{Y}_{n}$ the simplicial complex
    of cliques (also known as the \emph{flag complex} induced by the graph) in
    $\tilde{\Gamma}_{n}$. In particular, $\tilde{\Gamma}_{n}$ is
    the $1$-skeleton of $\tilde{Y}_{n}$.
    \begin{observation}\label{a4-c3-c13-dimension-of-brady-complex}
      All maximal chains in $\operatorname{NC}(n)$ have length $n$.
      Hence, all maximal simplices in $\tilde{Y}_{n}$ have
      dimension $n$.
    \end{observation}
    
    The most important fact about $\tilde{Y}_{n}$ is its contractibilty.
    \begin{theorem}[{\cite[Thm.~6.9 and Cor.~6.11]{a4-c3-c13-Brady2001}}]\label{a4-c3-c13-brady-main}
      The clique complex
      \(
        \tilde{Y}_{n}
      \)
      is contractible, and the braid group $\cB_{n}$ acts freely
      on it. Consequently, the orbit space
      \[
        Y_{n} :=
        \cB_{n}\setminus\tilde{Y}_{n}
      \]
      is a classifying space for the braid group $\cB_{n}$.\qed
    \end{theorem}

    \subsection{Higher generation by subgroups}
    \index{finiteness properties!higher generation by subgroups}
    For a subset
    \(
      I\subseteq
      \{\,
        1,\ldots,n
      \,\}
    \)
    let $\cB_{n}^{I}$ be the subgroup
    of $\cB_{n}=\pi_1(X_{n})$
    given by those paths, where the points in
    \(
      \{\,\,{v_{i}}\,\,|\,\,i\in I\,\,\}      
    \)
    do not move
    at all. For $k\in\{\,1,\ldots,n\,\}$, we put
    \(
      \cB_{n}^{k}
      :=
      \cB_{n}^{\{\,v_{k}\,\}}
    \),
    i.e., $\cB_{n}^{k}$ is the group of braids where
    the $k^{\text{th}}$ strand is rigid.
    It is, one might say, a group on $n-1$
    strands and one rod. However, since $v_{k}$ is a point on
    the boundary $\partial\, \mathbb{D}$, braiding with the rod is impossible.
    Thus, $\cB_{n}^{k}$ really is just an isomorphic
    copy of $\cB_{n-1}$ inside of $\cB_{n}$.
    Similarly,
    \(
      \cB_{n}^{I}
        =
        \bigcap_{{
          k\in I
        }}
          \cB_{n}^{k} 
    \)
    is isomorphic to
    $\cB_{n-\# I}$.

    Let $\operatorname{NC}^{k}(n)$ be the lattice of those
    non-crossing partitions in $\operatorname{NC}(n)$ where the singleton
    $\{\,k\,\}$ is a block. For a subset
    \(
      I\subseteq
      \{\,
        1,\ldots,n
      \,\}
    \),
    put
    \(
      \operatorname{NC}^{I}(n)
      :=
      \bigcap_{k\in I}
        \operatorname{NC}^{k}(n) 
    \).
    Then,
    \(
      \Gamma_{n}^{I}
      :=
      \{\,\,{
        \sigma_{p}
      }\,\,|\,\,
        p\in\operatorname{NC}^{I}(n)
      \,\,\}
    \)
    is a generating set for $\cB_{n}^{I}$.

    Note that the inclusion
    \(
      \cB_{n}^{I}
      \hookrightarrow
      \cB_{n}
    \)
    induces a bijection
    \(
      \Gamma_{n-\# I}
      \cong
      \Gamma_{n}^{I}
    \).
    Recall that $\Gamma_{n-\# I}$ is a poset
    with respect to divisibility. A priory, there are two poset structures
    on $\Gamma_{n}^{I}$: one from intrinsic divisibility with
    quotients again in $\Gamma_{n}^{I}$ and one induced
    from the ambient poset $\Gamma_{n}$, i.e., divisibility where
    quotients are allowed to be anywhere in $\Gamma_{n}$.
    However, since
    \(
      \Gamma_{n}^{I}
      =
      \Gamma_{n}
      \cap
      \cB_{n}^{I}
    \)%
    , the two
    poset structures coincide. Then,
    \(
      \Gamma_{n-\# I}
      \cong
      \Gamma_{n}^{I}
    \)
    is an isomorphism of posets.
    
    Moreover, the order preserving bijection
    \(
      \{\,1,\ldots,n-\# I\,\}
      \rightarrow
      \{\,1,\ldots,n\,\}\setminus I
    \)
    induces an isomorphism
    \(
      \operatorname{NC}(n-\# I)
      \cong
      \operatorname{NC}^{I}(n)
    \).
    This isomorphism is compatible with the poset isomorphism
    from Fact~\ref{a4-c3-c13-set}, and we have
    a commutative square of poset isomorphisms:
    \[
      \begin{CD}
        {\Gamma_{n-\# I}} @= {\Gamma_{n}^{I}} \\
        @|                                           @| \\
        {\operatorname{NC}(n-\# I)} @= {\operatorname{NC}^{I}(n)}
      \end{CD}
    \]

    The identity
    \(
      \Gamma_{n}^{I}
      =
      \Gamma_{n}
      \cap
      \cB_{n}^{I}
    \)
    has another consequence:
    \begin{observation}\label{a4-c3-c13-contractible-pieces}
      Let $\tilde{Y}_{n}^{I}$ be the full subcomplex
      spanned by $\cB_{n}^{I}$ as a set of vertices
      in $\tilde{Y}_{n}$. Then, $\tilde{Y}_{n}^{I}$
      is isomorphic to $\tilde{Y}_{n-\# I}$,
      whence it is contractible by Theorem~\ref{a4-c3-c13-brady-main}.
      For any
      coset $\beta\cB_{n}^{I}$, regarded
      as a set of vertices in $\tilde{Y}_{n}$, the full
      subcomplex spanned by
      $\beta\cB_{n}^{I}$
      is the translate $\beta\tilde{Y}_{n}^{I}$
      and also contractible.\qed
    \end{observation}
    \begin{observation}\label{a4-c3-c13-intersection-pattern}
      Assume that two coset complexes $\beta\tilde{Y}_{n}^{I}$
      and $\beta'\tilde{Y}_{n}^{J}$ intersect, say in
      $\bar{\beta}$. Then
      \(
        \beta\tilde{Y}_{n}^{I}
        =
        \bar{\beta}\tilde{Y}_{n}^{I}
      \)
      and
      \(
        \beta'\tilde{Y}_{n}^{J}
        =
        \bar{\beta}\tilde{Y}_{n}^{J}
      \).
      In this case, the intersection
      \[
        \bar{\beta}\tilde{Y}_{n}^{I}
        \cap
        \bar{\beta}\tilde{Y}_{n}^{J}
        =
        \bar{\beta}\tilde{Y}_{n}^{I\cup J}
      \]
      is contractible.
    \end{observation}
    Let $\cU:=(U_{\alpha})_{\alpha\in A}$ be a family
    of sets. For a subset $\sigma\subseteq A$ let
    \[
      U_{\sigma} :=
      \bigcap_{\alpha\in\sigma}U_{\alpha}
    \]
    denote the associated intersection.
    The simplicial complex
    \[
      N(\cU)
      :=
      \{\,\, \sigma\subseteq A \,\,|\,\,
        \varnothing \neq U_{\sigma}
      \,\,\}
    \]
    of all index sets whose associated intersection is non-empty is
    called the \emph{nerve} of the family $\cU$. If $\cU$
    is a family of subcomplexes in a CW complex, one has the following:
    \begin{theorem}[{Nerve Theorem, see \cite[Cor.~4G.3]{a4-c3-c13-Hatcher:AlgTop}}]
      \label{a4-c3-c13-nerve-cover}
      \index{theorem!nerve theorem}
      Suppose $\cU=(U_{\alpha})_{\alpha\in A}$
      is a covering of a simplicial complex $X$ by a family of
      contractible subcomplexes.
      Suppose further that, for each $\sigma\in N(\cU)$, the
      intersection $U_{\sigma}$ is contractible. Then, the nerve
      $N(\cU)$ is homotopy equivalent to $X$.
    \end{theorem}
    According to Observation~\ref{a4-c3-c13-intersection-pattern}, the Nerve Theorem
    applies in particular to the union:
    \[
      \tilde{X}_{n}
      :=
      \bigcup_{k}
        \bigcup_{{
          \beta\in\cB_{n}
        }}
          \beta\tilde{Y}_{n}^{k}          
    \]
    We deduce:
    \begin{proposition}\label{a4-c3-c13-BradySubcomplex}
      The complex $\tilde{X}_{n}$ is homotopy equivalent
      to the nerve $N$ of the family
      \[
        \{\,\,{ \beta\cB_{n}^{k} }\,\,|\,\,
          \beta\in\cB_{n},\,
          1\leqslant k\leqslant n
        \,\,\}
      \]
      of cosets.\qed
    \end{proposition}
    This relates to higher generation by subgroups as defined by
    Abels and Holz.
    \begin{definition}[{\cite[2.1]{a4-c3-c13-Abels.Holz}}]
      Let $G$ be a group and let $\mathfrak{H}$ be a family
      of subgroups. We say that $\mathfrak{H}$ is
      \emph{$m$-generating} for $G$ if the
      \emph{coset nerve}
      \[
        N_{G}(\mathfrak{H})
        :=
        N(
          \{\,\,gH\,\,|\,\,
            g\in G,\,
            H\in\mathfrak{H}
          \,\,\}
        )
      \]
      is $(m-1)$-connected.
    \end{definition}
    From Proposition~\ref{a4-c3-c13-BradySubcomplex}, we conclude immediately:
    \begin{corollary}\label{a4-c3-c13-first-generation-criterion}
      The family
      \(
        \mathfrak{B}_{n}
        :=
        \{\,
          \cB_{n}^{1},\ldots,
          \cB_{n}^{n}
        \,\}
      \)
      is $m$-generating for the braid group
      \(
        \cB_{n}
      \)
      if and only if
      $\tilde{X}_{n}$
      is $(m-1)$-connected.\qed
    \end{corollary}
    Recall that $\cB_{n}$ acts freely on
    the simplicial complex $\tilde{Y}_{n}$. The projection
    $\tilde{Y}_{n}\rightarrow Y_{n}$ is
    a covering space map. In fact, $\tilde{Y}_{n}$ is
    the universal cover of $Y_{n}$ and the
    braid group $\cB_{n}$ acts as the group of
    deck transformations. The subcomplex $\tilde{X}_{n}$
    is $\cB_{n}$-invariant. Let $X_{n}$
    be its image in $Y_{n}$.
    \begin{proposition}\label{a4-c3-c13-second-generation-criterion}
      The family
      \(
        \mathfrak{B}_{n}
        :=
        \{\,
          \cB_{n}^{1},\ldots,
          \cB_{n}^{n}
        \,\}
      \)
      is $m$-generating for the braid group
      \(
        \cB_{n}
      \)
      if and only if the
      pair $(Y_{n},X_{n})$
      is $m$-connected.
    \end{proposition}
    \begin{proof}
      First, consider the long exact sequence of homotopy groups for the
      inclusion $\tilde{X}_{n}\leq\tilde{Y}_{n}$:
      \[
        \cdots\longrightarrow
        \pi_{1}(\tilde{X}_{n}) \longrightarrow
        \pi_{1}(\tilde{Y}_{n}) \longrightarrow
        \pi_{1}(\tilde{Y}_{n},\tilde{X}_{n})
        \longrightarrow
        \pi_{0}(\tilde{X}_{n}) \longrightarrow
        \pi_{0}(\tilde{Y}_{n})
      \]
      Since $\tilde{Y}_{n}$ is contractible, we obtain isomorphisms:
      \[
        \pi_{d+1}(\tilde{Y}_{n},\tilde{X}_{n})
        \cong
        \pi_{d}(\tilde{X}_{n})
      \]
      On the other hand, $\tilde{Y}_{n}\rightarrow Y_{n}$
      is a covering space projection and therefore enjoys the homotopy lifting property.
      Moreover, $\tilde{X}_{n}$ is the full preimage of
      $X_{n}$. Therefore any map
      \[
        \left(
          \mathbb{B}^{d+1},
          \SSS^{d},
          * \right)
        \longrightarrow
        \left(
          Y_{n},
          X_{n},
          1 \right)
      \]
      lifts uniquely to a map
      \[
        \left(
          \mathbb{B}^{d+1},
          \SSS^{d},
          * \right)
        \longrightarrow
        \left(
          \tilde{Y}_{n},
          \tilde{X}_{n},
          1 \right)
      \]
      inducing a map
      \[
        \pi_{d+1}(
          Y_{n},
          X_{n}
        )
        \longrightarrow
        \pi_{d+1}(
          \tilde{Y}_{n},
          \tilde{X}_{n}
        )        
      \]
      which is inverse to the map
      \[
        \pi_{d+1}(
          \tilde{Y}_{n},
          \tilde{X}_{n}
        )
        \longrightarrow
        \pi_{d+1}(
          Y_{n},
          X_{n}
        )        
      \]
      coming from the covering space projection. Thus, we have isomorphisms
      \[
        \pi_{d+1}(
          Y_{n},
          X_{n}
        )
        \cong
        \pi_{d+1}(\tilde{Y}_{n},\tilde{X}_{n})
        \cong
        \pi_{d}(\tilde{X}_{n})
      \]
      and the claim follows from Corollary~\ref{a4-c3-c13-first-generation-criterion}.
    \end{proof}

    We can detect $1$-generating and $2$-generating families by hand.
    \begin{remark}\label{a4-c3-c13-small-gen}
      For $n\geqslant3 $, the family $\mathfrak{B}_{n}$ is
      $1$-generating for $\cB_{n}$, and for
      $n\geqslant 4$, it is $2$-generating.
    \end{remark}
    \begin{proof}
      A family $\mathfrak{H}$ is
      $1$-generating for $G$ if and only if
      $\bigcup_{H\in\mathfrak{H}}H$ generates
      $G$. It is $2$-generating for $G$ if
      $G$ is the product of the $H\in\mathfrak{H}$
      amalgamated along their intersections \cite[2.4]{a4-c3-c13-Abels.Holz}.

      Note that the braid group $\cB_{n}$ is generated
      by counter-clockwise rotations
      \[
        \beta_{ij}
        :=
        v_{i} \mapsto v_{j}
        \mapsto v_{i}
      \]
      around digons. Thus,
      \(
        \mathfrak{B}_{n}
        :=
        \{\,
          \cB_{n}^{1},\ldots,
          \cB_{n}^{n}
        \,\}
      \)
      generates as long as $n\geqslant 3$ since then each digon-generator
      is contained in some $\cB_{n}^{k}$.

      Considering the digon-generators for $\cB_{n}$,
      defining relations are given by braid relations, visible in
      isomorphic copies of $\cB_{3}$ inside $\cB_{n}$,
      and commutator relations, visible in isomorphic copies of
      $\cB_{4}$ inside $\cB_{n}$. Hence all necessary
      defining relations are visible in the amalgamated product of the
      $\cB_{n}^{k}\cong\cB_{n-1}$
      provided $n\geqslant 5$.

      For $n=4$, the challenge is to derive the commutator relations:
      \[
        \beta_{12}\beta_{34}=\beta_{34}\beta_{12}
        \qquad\text{and}\qquad
        \beta_{23}\beta_{41}=\beta_{41}\beta_{23}
      \]
      We do the first, the second is done analogously. Calculating with only
      three strands at a time, we find:
      \begin{align*}
        \beta_{12}\beta_{34}\beta_{24}
        &=
        \beta_{12}\beta_{23}\beta_{34}
        =
        \beta_{23}\beta_{13}\beta_{34}
        =
        \\
        &=
        \beta_{23}\beta_{34}\beta_{14}
        =
        \beta_{34}\beta_{24}\beta_{14}
        =
        \beta_{34}\beta_{12}\beta_{24}
      \end{align*}
      The desired commutator relation follows.
    \end{proof}
    \begin{remark}
      The little computation at the end of the preceeding proof shows
      that the commutator relations are redundant in the braid group
      presentation given in \cite[Lem.~4.2]{a4-c3-c13-Brady2001}. Accordingly,
      they are also redundant in the analoguous presentation from
      \cite[Prop.~2.1]{a4-c3-c13-BKL}.
    \end{remark}
      
    \begin{theorem}\label{a4-c3-c13-criterion-acyclic}
      For $n\geqslant 4$, the family
      $\mathfrak{B}_{n}$ is $m$-generating for
      $\cB_{n}$ if and only if the
      homology groups
      \(
        \operatorname{H}_{d}(
          Y_{n},X_{n}
        )
      \)
      are trivial for $1\leqslant d\leqslant m$.
    \end{theorem}
    \begin{proof}
      As $n\geqslant 4$, the pair
      $(
        Y_{n},X_{n}
      )$
      is $1$-connected by
      Propositions~\ref{a4-c3-c13-second-generation-criterion} and~\ref{a4-c3-c13-small-gen}.
      Thus, it follows from the relative Hurewicz theorem
      that $m$-connec\-tiv\-ity of the pair is equivalent to
      $m$-acyclicity. By Proposition~\ref{a4-c3-c13-second-generation-criterion},
      this translates into higher generation of
      $\cB_{n}$ by $\mathfrak{B}_{n}$.
    \end{proof}
    As the pair
    $(
        Y_{n},X_{n}
    )$
    consists of finite complexes that can be described explicitly,
    Theorem~\ref{a4-c3-c13-criterion-acyclic} implies that
    is it a \emph{finite problem} to determine
    the higher connectivity properties
    of $\cB_{n}$ relative to the family $\mathfrak{B}_{n}$.
    In particular, the question whether the bounds derived in
    \cite[Example 15.5.4]{sfb701}
    for higher generation in braid groups are sharp becomes amenable to
    empirical investigation.

    \subsection{Curvature in braid groups}\label{a4-c3-c13-curvature-in-braid-groups}
      \begin{definition}\label{a4-c3-c13-artin-group}
        For an $n \times n$ symmetric matrix
        $(m_{ij})$ with entries in \linebreak
        $\{ 2,3, \dots \} \cup \{ \infty\}$ we define the associated
        \emph{Artin group} to be
        \[
          \left\langle\,{
            s_{1},\ldots,s_{n}
          }\,\,\,\vrule\,\,
            \underbrace{
              s_{i}s_{j}s_{i}\cdots
            }_{m_{ij}\text{\ factors}}
            =
            \underbrace{
              s_{j}s_{i}s_{j}\cdots
            }_{m_{ij}\text{\ factors}}
          \,\right\rangle
        \]
        Here, $m_{ij}=\infty$ indicates that there is no defining
        relation for $s_{i}$ and $s_{j}$. We will refer to the
        relations appearing above as \emph{braid relations} (even
        though some authors reserve this term for the relation with
        $m_{ij}=3$).
        
        If one additionally forces the generators $s_{i}$ into being
        involutions, one obtains the associated \emph{Coxeter
          group}. A pair consisting of a Coxeter group together with
        the generating set $\{ s_{1},\ldots,s_{n} \}$ is called a
        \emph{Coxeter system}; its \emph{rank} is defined to be the
        cardinality of the generating set. If the Coxeter group is spherical, 
 the Coxeter system is said to be \emph{spherical} as well.
        
A Coxeter group is \emph{spherical} if it is finite; an Artin group is
\emph{spherical} if the corresponding Coxeter group is spherical.
      \end{definition}
      Note that the braid group $\cB_{n}$ is an Artin group
      and the symmetric group $S_{n}$ is the associated Coxeter
      group. Here, $m_{ij}=3$ for 
      $|i-j|=1$
      and $m_{ij}=2$ otherwise.
      See Fact~\ref{a4-c3-c13-artin-presentation-braid-groups}
      
      Artin groups form a rich class of groups of importance in geometric group
      theory and beyond. From geometric group theory perspective they remain
      in focus largely due to the following conjecture.
      \begin{conjecture}[Charney]\label{a4-c3-c13-conj charney}
        \index{conjecture!Charney's}
        Every Artin group is {\small{CAT(0)}}, i.e. it acts properly and
        cocompactly on a {\small{CAT(0)}} space.
      \end{conjecture}
      A {\small{CAT(0)}} space is a metric space with curvature bounded from
      above by 0; for details see the book by
      Bridson--Haefliger~\cite{a4-c3-c13-BridsonHaefliger1999}. From the current
      perspective let us list some properties of {\small{CAT(0)}} groups:
      algorithmically, such groups have
      quadratic Dehn functions and hence soluble word problem; geometrically,
      all free-abelian subgroups thereof are undistorted; algebraically,
      the centralisers of infinite
      cyclic subgroups thereof split; topologically, the space witnessing
      {\small{CAT(0)}}-ness of a group $G$ is a finite model for
      $\underline{E}G$ and
      thus, for example, allows to compute the K-theory of the reduced
      $C^{*}$-algebra $C^{*}_{\text{r}}(G)$ provided the Baum--Connes conjecture is
      known for $G$.
      
      Conjecture~\ref{a4-c3-c13-conj charney} has been verified by
      Charney--Davis for \emph{right-angled Artin groups (RAAGs)},
      that is for Artin groups with each $m_{ij}$ equal to 2 or
      $\infty$. Outside of this class, the conjecture is mostly open.
      In particular, it is open (in general) for the braid groups
      $\cB_{n}$.

      To prove that a group $G$ is {\small{CAT(0)}}, one has to first
      construct a space $X$ on which $G$ acts properly and
      cocompactly, and then prove that the space is indeed {\small{CAT(0)}}.
      We shall use the space $\tilde{Y}_{n}$ from above, on
      which $\cB_{n}$ acts freely and with compact quotient.
      
      What is missing, however, is a metric structure on
      $\tilde{Y}_{n}$.  Such a metric can be specified by realising
      the simplices in euclidean space, i.e., by endowing each simplex
      in $\tilde{Y}_{n}$ with the metric of a euclidean
      polytope. Instead of the standard one, we will follow
      Brady--McCammond~\cite{a4-c3-c13-BradyMcCammond2010}.
      
      \begin{definition}
        Let $e_{1},\ldots,e_{m}$ denote the
        standard basis of $\RR^{m}$.
        The \emph{$m$-orthoscheme} is the convex hull
        of
        \(
          \{\,
            0,
            e_{1},
            e_{1}+e_{2},
            \ldots,
             e_{1}+e_{2}+\cdots+e_{m}
          \,\}
        \).
        The orthoscheme has the structure of an $m$-simplex
        and the vertices come with a grading: the vertex
        \(
          e_{1} + \cdots + e_{k}
        \)
        is declared to be of \emph{rank} $k$. 
      \end{definition}
      
      We now endow each maximal simplex in $\tilde{Y}_{n}$
      with the orthoscheme metric. Let
      \[
        \Sigma
        =
        \{\,
          \beta, \beta\sigma_{1}, \ldots,
          \beta\sigma_{n}
        \,\}
      \]
      be a maximal simplex.
      Here, $\beta$ is a braid in $\cB_{n}$ and
      \(
        1 < \sigma_{1}
        <\sigma_{2}
        <\cdots<
        \sigma_{n}
      \)
      is a maximal chain in $\Gamma_{n}\cong\operatorname{NC}(n)$,
      which has length $n$ by
      Observation~\ref{a4-c3-c13-dimension-of-brady-complex}.
      We endow $\Sigma$ with the metric of the standard
      $n$-orthoscheme by identifying $\beta\sigma_{k}$
      with the vertex of rank $k$ in the orthoscheme.
      \index{orthoscheme metric}
      It is easy to see that if two maximal simplices intersect,
      they induce identical metric on their common face. Thus we have
      turned $\tilde{Y}_{n}$ into a metric simplicial complex.
      
      Note that $\tilde{Y}_{n}$ is obtained by gluing copies of a
      single shape, the $n$-orthoscheme, and so $\tilde{Y}_{n}$ is a
      geodesic metric space by a result of Bridson (finitely many
      shapes of cells would suffice). Since the shape is euclidean, we
      may use Gromov's link condition and deduce the following:
      \begin{lemma}
        $\tilde{Y}_{n}$ is {\small{CAT(0)}} if and only if the link
        of each vertex in $\tilde{Y}_{n}$ is {\small{CAT(1)}}.
      \end{lemma}
      Here {\small{CAT(1)}} means that the curvature of the space is bounded
      above by that of the unit sphere; again, for details
      see~\cite{a4-c3-c13-BridsonHaefliger1999}.

      The poset $\Gamma_{n}\cong\operatorname{NC}(n)$ has a
      unique maximal element, which is the braid $\gamma$ corresponding
      to the full counter-clockwise rotation:
      \[
        v_{1}\mapsto v_{2}\mapsto\cdots\mapsto
        v_{m}\mapsto v_{1}
      \]
      The $n^{\textit{th}}$ power $\gamma^{n}$ is central in the braid
      group $\cB_{n}$. In fact, it generates the infinite cyclic
      center of $\cB_{n}$.  Brady--McCammond observed
      in~\cite{a4-c3-c13-BradyMcCammond2010} that this algebraic fact
      has a geometric counterpart: $\tilde{Y}_{n}$ splits as a
      cartesian product of the real line $\RR$ and another metric
      space. The $\RR$-factor inside $\tilde{Y}_{n}$ points in the
      direction of the edges labelled by $\gamma$.
      
      Because of this, instead of looking at the link of a vertex
      $u$ in $\tilde{Y}_{n}$, one can look at
      the link of a midpoint of the (long) edge
      $(u,u\gamma)$; every two such links
      are isometric (since $\cB_{n}$ acts transitively
      on the vertices of $\tilde{Y}_{n}$), and so let
      $L$ denote any such link.
      
      To compute the curvature of $L$, it is enough to study
      the subcomplex of $\tilde{Y}_{n}$ spanned by all
      simplices containing the edge $(u,u\gamma)$.
      Clearly, this is the subcomplex spanned by $L$ and $u\sigma$
      with
      \(
        1\leq\sigma\leq\gamma
      \),
      with simplices defined by the chain condition as before. Thus, such
      a link is isomorphic as a simplicial complex to the realisation
      of $\operatorname{NC}(n)$; the subcomplex also comes with a
      metric, and it is clear that this coincides with the realisation
      of $\operatorname{NC}(n)$ being endowed with its own orthoscheme
      metric defined as before by identifying each maximal simplex with
      the $n$-orthoscheme. We will refer to the realisation
      of $\operatorname{NC}(n)$ with this metric simply as the
      \emph{orthoscheme complex} of $\operatorname{NC}(n)$.
      
      Note that if the orthoscheme complex of $\operatorname{NC}(n)$
      is {\small{CAT(0)}}, then $L$, isometric to the link of the
      midpoint of the main diagonal, is $\mathrm{CAT(1)}$, which implies
      that $\tilde{Y}_{n}$, and so $\cB_{n}$,
      is {\small{CAT(0)}}.
      
      In view of the above, Brady--McCammond formulate the following
      conjecture.
      \begin{conjecture}[{\cite[Conj.~8.4]{a4-c3-c13-BradyMcCammond2010}}]
        For every $n$, the orthoscheme complex of
        $\operatorname{NC}(n)$ is {\small{CAT(0)}}, and so the
        braid group $\cB_{n}$ is {\small{CAT(0)}}.
      \end{conjecture}
      For $n\leqslant 4$, the conjecture is easily seen to be true. 
      
      If we know that the orthoscheme complexes of $\operatorname{NC}(m)$
      are {\small{CAT(0)}} for each $m<n$, then in fact the
      orthoscheme complex of $\operatorname{NC}(n)$ is {\small{CAT(0)}} if and
      only if the link $L$ is {\small{CAT(1)}}. Thus, for $n=5$,
      it is enough to study $L$, which is the realisation of the
      poset obtained from $\operatorname{NC}(n)$ by removing the trivial
      and improper partitions, and endowing the realisation with the
      \emph{spherical orthoscheme metric}. Knowing that the conjecture
      is true for all $m<5$ tells us that $L$ is
      locally {\small{CAT(1)}}. Thus, using the work of
      Bowditch~\cite{a4-c3-c13-Bowditch1995}, it is enough to check whether
      any loop in $L$ of length less than $2\pi$ can be
      \emph{shrunk}, i.e., homotoped to the trivial loop without
      increasing its length in the process.
      
      Brady--McCammond use a computer to analyse all loops in
      $L$ shorter than $2\pi$, and show that they are indeed
      shrinkable, thus establishing:
      \begin{theorem}[{\cite[Thm.~B]{a4-c3-c13-BradyMcCammond2010}}]
        For $n\leqslant 5$, the braid group $\cB_{n}$
        is {\small{CAT(0)}}.
      \end{theorem}
      
      Haettel, Kielak and Schwer go beyond that, proving
      \begin{theorem}[{\cite[Cor.~4.18]{a4-c3-c13-Haetteletal2016}}]\label{a4-c3-c13-thm c3-HKS}
        For $n\leqslant 6$, the braid group $\cB_{n}$ is {\small{CAT(0)}}.
      \end{theorem}
      Note that their proof is not computer assisted. The crucial improvement
      in the work of Haettel--Kielak--Schwer is to use the observation
      (present already in~\cite{a4-c3-c13-BradyMcCammond2010}), that the link
      $L$ can be embedded into a spherical building, in the following
      way.
      
      First observe that the vertices of $L$ are non-trivial proper
      partitions; let $p$ be such a partition with blocks
      \(
        B_{1},\ldots,B_{k}
      \).
      Let $\FF$ be the field of two elements; we associate to
      $p$ the subspace of
      \(
        \FF^{n}
        =
        \langle 
          \boldsymbol{b}_{1},\ldots,\boldsymbol{b}_{n}
         \rangle
      \)
      which is the intersections of the kernels of the characters
      \[
        \sum_{{
          j \in B_{i}
        }}
          \boldsymbol{b}^{*}_{j} 
        = 0
      \]
      where $1 \leqslant i \leqslant k$, and
      $\boldsymbol{b}^{*}_{j}$ is the $j$-th character in the basis
      dual to the $\boldsymbol{b}_{j}$.
      
      It is easy to see that this gives a map sending each vertex of
      $L$ to a proper non-trivial subspace of
      \(
        V
        :=
        \ker\big(
          \sum_{j=1}^{n}
            \boldsymbol{b}^{*}_{j} 
        \big)
      \).
      But these subspaces are precisely the vertices of the spherical
      building of
      \(
        \operatorname{SL}_{n-1}(\FF)
      \),
      and it turns out that
      our bijection extends to a map sending each maximal simplex
      in $L$ onto a chamber (i.e. maximal simplex) in the
      building in an isometric way. Thus we may view $L$ as
      a subcomplex of the building.
      
      The spherical building is {\small{CAT(1)}}, and this information gives
      the extra leverage used to prove Theorem~\ref{a4-c3-c13-thm c3-HKS}.

 \section{Non-crossing partitions in Coxeter groups}
 
 In this section, we introduce the general theory of non-crossing
 partitions and explain how non-crossing partitions appear in group
 theory.  As already observed in the beginning of
 Section~\ref{a4-c3-c13-curvature-in-braid-groups}, the symmetric
 group $S_n$ is a Coxeter group and $(S_n,S_{\rm tr})$ is a Coxeter
 system of rank $n-1$ where
 $$S_{\rm tr}:= \{(i,i+1)~|~ 1 \leq i \leq n-1\}$$
 is the set of neighbouring transpositions.
 
 Every Coxeter system $(W,S)$ acts faithfully on a real vector space
 that is equipped with a symmetric bilinear form $(- , -)$ such that
 for every $s \in S$ there is a vector $\alpha_s \in V$ so that $s$
 acts as the reflection
 $$r_{\alpha_s}: v \mapsto v - 2\frac{(v, \alpha_s)}{(\alpha_s,
   \alpha_s)}\alpha_s$$
 on $V$. Thus every Coxeter group is a reflection group that is a
 group generated by a set of reflections on a vector space
 $(V,(- , -))$.
 
 The vectors $\alpha_s$ can be chosen so that the subset
 $\Phi = \{w(\alpha_s) ~|~s \in S, w \in W\}$ of $V$ is a so called
 \emph{root system}.  For a spherical Coxeter system a \emph{root
   system}\index{root system} $\Phi$ is characterised by the following
 three axioms
 
 \begin{itemize}
 	\item[(R1)]  $\Phi$ generates $V$;
 	\item[(R2)] $\Phi \cap \RR\alpha = \{\pm \alpha\}$ for all $\alpha \in \Phi$;
 	\item[(R3)] $s_\alpha(\beta)$ is in $\Phi$ for all $\alpha, \beta \in \Phi$. 
 \end{itemize}
 The spherical Coxeter groups $W$ are precisely the finite real
 reflection groups.
 
 Coxeter classified the finite root systems which then also gives a
 classification of the spherical Coxeter systems: there are the
 infinite families of type $A_n, B_n, C_n$ and $D_n$ and some
 exceptional groups. For instance $(S_n,S_{\rm tr})$ is of type
 $A_{n-1}$.  Note that the groups of type $B_n$ and $C_n$ are
 isomorphic; and also that the root systems of type $A_n, B_n, C_n$
 and $D_n$ are all \emph{crystallographic}\index{root system!crystallographic} that is
 $$\frac{(\alpha,\beta)}{(\alpha, \alpha)} \in \ZZ~\mbox{for all}~
 \alpha, \beta \in \Phi.$$
 
 We call $T:= \cup_{w \in W}w^{-1}Sw$ the set of reflections of the
 Coxeter system $(W,S)$.  If the system is spherical, then $T$ is
 indeed the set of all reflections.
 
 For instance in the symmetric group $S_n$ the set $T$ is the
 conjugacy class of transpositions, see also
 Section~\ref{a4-c3-c13-braid-groups}. There the so called
 \emph{absolute order}\index{absolute order} $\leq_T$ on $S_n$ has
 been introduced. Let $[id, (1,2,\ldots ,n)]_{\leq_T}$ be the closed
 intervall in $S_n$ with respect to $\leq_T$.  In Fact 1.3.4 it has
 been stated that $(\operatorname{NC}(n),\subseteq)$ and
 $([id, (1,2,\ldots ,n)]_{\leq_T}, \leq_T)$ are posets that are
 isomorphic.  Therefore $\operatorname{NC}(n)$ can be thought of being
 of type $A_{n-1}$.
 
 Out of combinatorial interest, Reiner generalised the concept of non-crossing partitions to  the infinite series of type $B_n$ and $D_n$
 geometrically \cite{a4-c3-c13-Rei}.  Independently of his work and of each other Brady and Watt  \cite{a4-c3-c13-BrWa}  as well as Bessis \cite{a4-c3-c13-Dual} 
 generalised the concept of non-crossing partitions to all the finite Coxeter systems. Their approach agrees with Reiner's in type $B_n$ \cite{a4-c3-c13-Arm}.
 
 Brady and Watt as well as Bessis started independently the study of
 the dual Coxeter system $(W,T)$ instead of $(W,S)$.  A \emph{dual
   Coxeter system}\index{dual Coxeter system} $(W,T)$ of finite rank
 $n$ has the property that there is a subset $S$ of $T$ such that
 $(W,S)$ is a Coxeter system \cite{a4-c3-c13-Dual}.  It then follows
 that $T$ is the set of reflections in $(W,S)$. This concept is called
 by Bessis \emph{dual approach to Coxeter and Artin groups}.
 
 A \emph{(parabolic) standard Coxeter element}\index{Coxeter element}
 in $(W, S)$ is the product of all the elements in (a subset of) $S$
 in some order and a \emph{(parabolic) Coxeter element} in $(W,T)$ is
 a (parabolic) standard Coxeter element in $(W,S)$ for some simple
 system $S$ in $T$ for $W$.
 
 For instance in type $A_{n-1}$, so in the symmetric group $S_n$, the
 standard Coxeter elements with respect to $S= S_{\rm tr}$ are
 precisely those $n$-cycles in $S_n$ that can be written as a first
 increasing and then decreasing cycle. All the $n$-cycles in $S_n$ are
 the Coxeter elements in the dual system $(S_n, T)$ where $T$ is the
 set of reflections, that is the conjugacy class of transpositions.

 The partial order $\leq_T$ on the symmetric group $S_n$ presented in
 Section~\ref{a4-c3-c13-braid-groups} can be generalized to all the
 dual Coxeter systems $(W,T)$. We consider the Cayley graph
 ${\rm CG}_T(W)$ of the group $W$ with respect to the generating set
 $T$.  For $u,v \in W$ we declare $u \leq_T v$ if there is a geodesic
 path in the Cayley graph connecting the identity to $v$ and passing
 through $u$. This partial order is also called the \emph{absolute
   order} on $W$.
 
 We also introduce a length function $l_T$ on $W$: for $u \in W$ we
 define $l_T(u) = k$ if there is a geodesic path from the identity to
 $u$ of length $k$ in the Cayley graph. Notice, if $l_T(u) = m$ then
 $u$ is the product of $m$ reflections, that is $u = t_1\cdots t_m$
 with $t_i \in T$, and there is no shorter factorisation of $u$ in a
 product of reflections. In this case we say that $u = t_1\cdots t_m$
 is a \emph{$T$-reduced factorisation}\index{factorisation!$T$-reduced}
 of $u$.  In particular, if $u \leq_T v$, then there are
 $k, m \in \NN$ with $k \leq m$ and reflections $t_1, \ldots , t_m$ in
 $T$ such that $u = t_1 \cdots t_k$ and $v= t_1 \cdots t_m$. Thus
 $$u \leq_T v~\mbox{if and only if}~l_T(u)+l_T(u^{-1}v) = l_T(v).$$ 
 
 \begin{definition}
   For a dual Coxeter system $(W,T)$ and a Coxeter element $c$ in $W$
   the set of \emph{non-crossing
     partitions}\index{partition!non-crossing partition!in a Coxeter system} is
 	$$\operatorname{NC}(W,c) = \{u \in W~|~u \leq_T c\}.$$
 \end{definition}
 
 This definition is conform with the definition in type $A_n$,
 see~Fact~\ref{a4-c3-c13-set}.
 
 The length function $l_T$ yields a grading on $\operatorname{NC}(W,c)$ and the map
 $$d: \operatorname{NC}(W,c) \rightarrow \operatorname{NC}(W,c),~ x \mapsto x^{-1}c$$
 a duality on $\operatorname{NC}(W,c)$ that inverses the order
 relation.
 
 This implies the following.
 \begin{fact}
 	$\operatorname{NC}(W,c)$ is a poset that is 
 	\begin{itemize}
 		\item graded
 		\item selfdual
 		\item ~\cite{a4-c3-a4-c3-c13-BrWa2,a4-c3-c13-Dual} a
                  lattice if $W$ is spherical.
 	\end{itemize}
 \end{fact}
 
 The number of elements in $\operatorname{NC}(W,c)$ in a finite dual Coxeter system  of type $X$ 
 is the generalised Catalan number\index{Catalan number!generalised Catalan number} of type $X$.  In types $B_n$ and $D_n$ there are
 also nice geometric models for the posets of non-crossing partitions.
 
 Note that in a spherical Coxeter system always
 $T \subseteq \operatorname{NC}(W,c)$.
 
 There is also a presentation of $W$ with generating set $T$
 \cite{a4-c3-c13-Dual}.  The relations are the so called \emph{dual
   braid relations} with respect to a Coxeter element $c \in W$:
 $$
\mbox{ for every}~ s,t,t^\prime \in T~\mbox{set}~ st = t^\prime s~\mbox{  whenever}
$$ 
$$\mbox{ the relation }~st = t^\prime s ~\mbox{holds in $W$ and}~ st \leq_T c.
$$
  
 The Matsumoto property means if we have for some $w \in W$ two
 shortest factorisations as products of elements of $S$, or
 equivalently two geodesic paths from $id$ to $w$ in the Cayley graph
 ${\rm CG}_S(W)$, then we can transform one factorisation or path into
 the other one just by applying braid relations; that is $W$ has a
 group presentation as given in
 Definition~\ref{a4-c3-c13-artin-group}.
 
 The \emph{dual Matsumoto property}\index{dual Matsumoto property} for
 a Coxeter element $c \in W$ is the statement that if we have two
 shortest factorisations
 $$c = t_1 \cdots t_m = u_1 \cdots u_m ~\mbox{with}~ t_i, u_i \in T$$
 as products of elements of $T$, that is two $T$-reduced
 factorisations of $c$ in $W$, then one factorisation can be
 transformed into the other one just by applying dual braid relations.
 It follows that the dual Matsumoto property holds for $c$, since
 $$\langle T~|~~\mbox{dual braid relations}\rangle$$ is a presentation of $W$.
 
 We obtain the dual Matsumoto property for an arbitrary element
 $w \in W$ by replacing $c$ by $w$ in the definition of the dual braid
 relations and of the dual Matsumoto property above.
 
 For an element $w \in W$, let  
 $$\operatorname{Red}_T(w) = \{(t_1, \ldots, t_m)~|~t_i \in T~\mbox{and}~ w = t_1 \cdots t_m~\mbox{is $T$-reduced}\}.$$
 
 The dual Matsumoto property for $w \in W$ is equivalent to the
 transitive \emph{Hurwitz action}\index{Hurwitz action} of the braid
 group $\cB_{l_T(w)}$ on the set of $T$-reduced factorisations
 $\operatorname{Red}_T(w)$ of $w$.  For the braid
 $\sigma_i \in \cB_{l_T(w)}$, see
 Fact~\ref{a4-c3-c13-artin-presentation-braid-groups}, the action is
 given by
 $$\sigma_i(t_1, \ldots , t_n) = (t_1, \ldots ,t_{i-1}, t_i^{-1}
 t_{i+1}{t_i}, t_i, t_{i+2}, \ldots , t_n).$$
 We will discuss this action in more detail in the next section.
 
 The dual approach can also be applied to Artin groups; given a
 Coxeter system $(W,S)$, we will denote the corresponding Artin group
 by $\cA(W,S)$. If in the following the Coxeter system $(W,S)$ is of
 type $X$, then we abbreviate $\cA(W,S)$ either by $\cA(W)$ or by
 $\cA_X$.  Further we take a copy $S_a$ of $S$ in $\cA(W,S)$ and write
$$
\cA(W,S) := \langle S^{}_a~|~ (s^{}_1)^{}_a(s^{}_2)^{}_a (s^{}_1)^{}_a 
\cdots = (s^{}_2)^{}_a (s^{}_1)^{}_a (s^{}_2)^{}_a 
\cdots ~\mbox{for}~s^{}_1, s^{}_2 \in S \rangle
$$  
in order to distinguish between $W$ and $\cA(W)$. We call an Artin group 
$\cA (W)$ \emph{spherical} if the Coxeter group is spherical.  
 And in the rest of this section, we always consider spherical Artin groups.
 
 Notice that the Matsumoto property implies that one can lift every
 $w \in W$ to an element in $\cA(W)$ just by mapping $w$ to
 $(s_1)_a \cdots (s_k)_a \in \cA_W$ whenever $w = s_1 \cdots s_k$ is a
 reduced factorisation of $w$ into elements of $S$. We denote this
 section of $W$ in $\cA(W)$ by $\cW$.
 
 The non-crossing partitions are a good tool for the better
 understanding of the spherical Artin groups; for instance they can be
 used to construct a finite simplicial classifying space for the
 spherical Artin groups
 (see~Section~\ref{a4-c3-c13-classifingy-space-for-braid-group}), or
 to solve the word or the conjugacy problem in them, see
 \cite{a4-c3-c13-BrWa,a4-c3-c13-Dual}.
 
 The basic idea of this solution of the word and the conjugacy problem
 in the spherical Artin group $\cA(W)$ is to give a new presentation
 of $\cA(W)$ as follows.  Let $\operatorname{NC}(W,c)_a$ be a copy of
 the set of non-crossing partitions $\operatorname{NC}(W,c)$ with
 respect to a standard Coxeter element $c$, that is there is a
 bijection
 $$a: \operatorname{NC}(W,c) \rightarrow \operatorname{NC}(W,c)_a.$$
 Then the new generating set is $\operatorname{NC}(W,c)_a$; and the
 new relations are the expressions $(w_1)_a \cdots (w_r)_a$ whenever
 $w_1, w_2, \ldots , w_r$ are the vertices of a circuit in
 $$[id, c]_{\leq_T} \subseteq {\rm CG}_{\operatorname{NC}(W,c)}(W).$$
 Then this presentation can be used to obtain a new normal form for
 the elements in $\cA(W)$ \cite{a4-c3-c13-Dual}. Notice that this
 presentation generalises the presentation of the braid group given by
 Birman, Ko and Lee \cite{a4-c3-c13-BKL} to all the spherical Artin
 groups, see also Fact~\ref{a4-c3-c13-triangular-presentation} in
 Section~\ref{a4-c3-c13-classifingy-space-for-braid-group}.
 
 Next, we explain this new presentation. Denote the group given by the presentation above by $\cA(W,c)$.
 The strategy to prove that $\cA(W,c)$ and $\cA(W)$ are isomorphic is to use Garside theory.
 As a first step the presentation above can be transformed into a presentation 
 with set of generators a copy $T_a = \{t_a~|~t \in T\}$ of $T$ and set of relations the dual braid relations with respect to $c$.
 The next step is to consider the monoid $\cA(W,c)^*$ 
 generated by $T_a$ and the dual braid relations, and to show that this is a Garside monoid.
 Then using Garside theory  one shows that the group of fractions $\mathrm{Frac}(\cA(W,c)^*)$ of $\cA(W,c)^*$  equals $\cA(W,c)$.
 The last step is to prove that the group of fractions $\mathrm{Frac}(\cA(W,c)^*)$ and the Artin group $\cA(W)$ are isomorphic.
 
 \begin{theorem}[{\cite{a4-c3-c13-Dual}}]\label{a4-c3-c13-NewPresentation}
 	Let $\cA_W$ be a spherical Artin group. Then,
 	$$\cA_W \cong \langle T_a~ |~ t_a t'_a= (tt't)_a t_a~\text{if } t, t' \in T~\mbox{and}~ tt'\leq_T c\rangle.$$
 \end{theorem}

 Note also that a basic ingredient in the proof of
 Theorem~\ref{a4-c3-c13-NewPresentation} is the dual Matsumoto
 property for $c$, that is the transitivity of the Hurwitz action of
 the braid group $\cB_{l_T(c)}$ on $\operatorname{Red}_T(c)$.
  
 The isomorphism between $\cA(W,c)$ and $\cA_W$ given by Bessis is
 difficult to understand explicitly.  So an immediate question is what
 the elements of $\operatorname{NC}(W,c)_a$ are expressed in the
 generating set $S_a$?
 
 The rational permutation braids, that is, the elements ${ x}{y}^{-1}$
 where ${x}, {y} \in \cW$, are also called \emph{Mikado
   braids}\index{Mikado braid} as they satisfy in type $A_{n-1}$ a
 topological condition and are therefore easy to recognise. This
 condition on an element in the Artin group $\cA(W)$ of type
 $A_{n-1}$, that is on a braid in the braid group $\cB_n$, is that we
 can lift and remove continuously one strand after the next of the
 braid without disturbing the remaining strands until we reach an
 empty braid \cite{a4-c3-c13-DG}.
 
 \begin{theorem}\label{a4-c3-c13-Mikado}
   If $\cA_W$ is spherical Artin group and $c \in W$ a standard
   Coxeter element, then the dual generators of $\cA(W,c)$, that is
   the elements of $\operatorname{NC}(W,c)_a$, are Mikado braids in
   $\cA_W$.
 \end{theorem}
 \begin{proof}
   This is \cite{a4-c3-c13-DG} for those groups of type different from
   $D_n$ and \cite{a4-c3-c13-BG} for those of type $D_n$.
 \end{proof}
 
 Notice that Licata and Queffelec \cite{a4-c3-c13-LQ} have a proof of
 Theorem~\ref{a4-c3-c13-Mikado} in types A,D,E with a different
 approach using categorification.
 
 In order to be able to find a topological property that characterises
 the Mikado braids as in type $A_{n-1}$ topological models for the
 series of spherical Artin groups $\cA_W$ are needed.  There is an
 embedding of Artin groups of type $B_n$ into those of type
 $A_{2n-1}$.  The situation in type $D_n$ is as follows
 \cite{a4-c3-c13-BG}: The root system of type $D_n$ embeds into the
 root system of type $B_n$, which implies that the Coxeter system of
 type $D_n$ is a subsystem of that one of type $B_n$. But there is not
 an embedding of the Artin group of type $D_n$ into that one of type
 $B_n$ that satisfies a certain natural condition.  Let $(W, S)$ be a
 Coxeter system of type $B_n$. Then there is precisely one element
 $s \in S$ that is a reflection corresponding to a short root. Let
 $$\overline{\cA_{B_n}} := \cA_{B_n}/\ll s^2 \gg,$$
 where $\ll s^2 \gg$ is the normal closure of $s^2$ in  $\cA_{B_n}$.
 Then the following holds.
 
 \begin{proposition}[{\cite[Lem.~2.5 and Prop.~2.7]{a4-c3-c13-BG}}]
   There is a natural embedding of $\cA_{D_n}$ onto an index-$2$
   subgroup of $\overline{\cA_{B_n}}$. More precisely, there is the
   following commutative diagram
 	\[
 	\begin{CD}
          {\cA_{B_n}} @>\pi>> {\overline{\cA_{B_n}}}   @<<< {\langle t_1, \dots, t_{n}\rangle} @<\cong<< {\cA_{D_n}}\\
          @|                    @VV\pi_{\overline\cB}V @V{\pi_\cD}VV   \\
          {\cA_{B_n}} @>>\pi_\cB> {W_{B_n}} @<<< {W_{D_n}}
 	\end{CD}
 	\]
 \end{proposition} 
 
 The embedding of $\cA_{D_n}$ into $\overline{\cA_{B_n}}$ makes it
 possible to associate braid pictures to the $\cA_{D_n}$-elements and
 to characterise Mikado braids in type $D_n$ geometrically.
 
 \begin{figure}[h!]
 	%
 	%
 	%
 	%
 	%
 	%
 	%
 	%
 	%
 	%
 	%
 	%
 	%
 	%
 	%
 	%
 	%
 	%
 	%
 	%
 	%
 	%
 	%
 	%
 	%
 	%
 	%
 	\begin{center}
 		\includegraphics[scale=0.5]{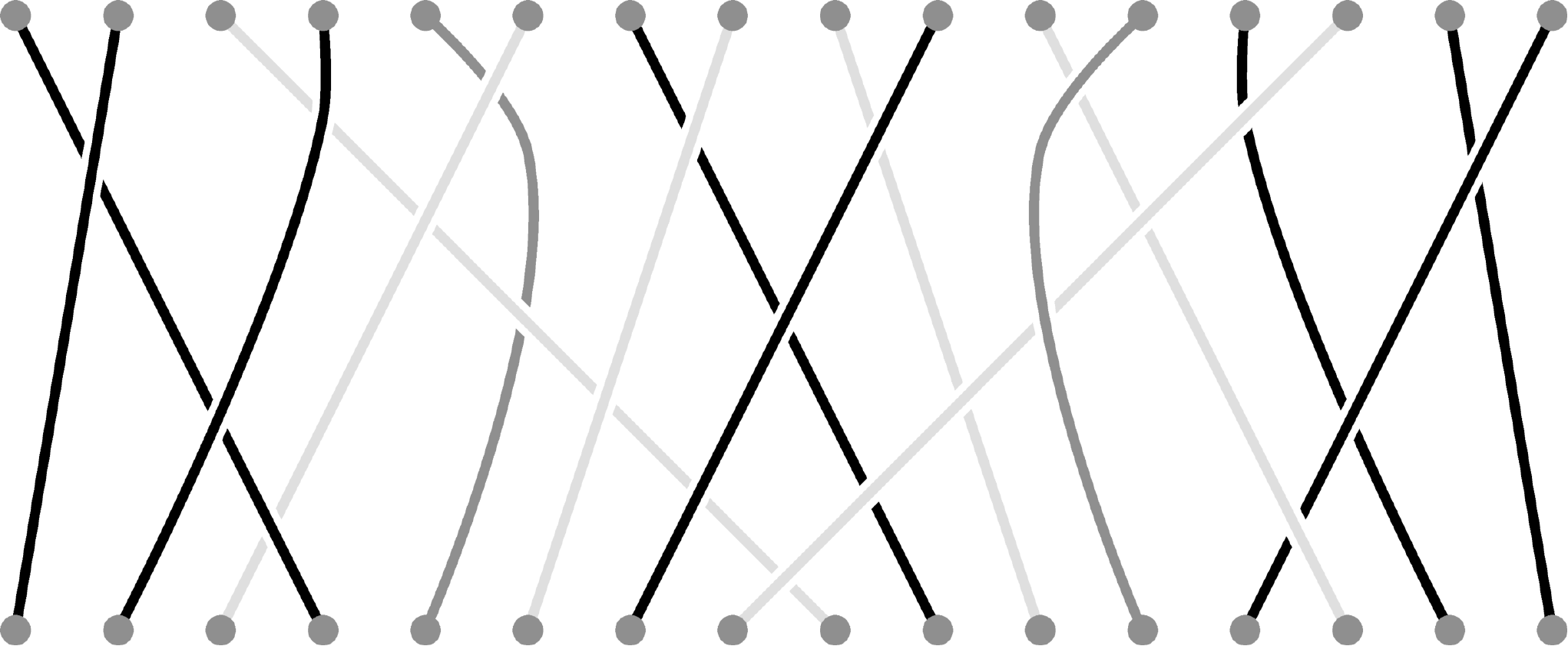}
 	\end{center}
 	\caption{A Mikado braid in $\cA_{B_8}$ whose image 
 		in $\overline{\cA_{B_8}}$ is a Mikado braid in
 		$\cA_{D_8}$.}
 \end{figure}
 
 A reader familiar with Hecke algebras will find it interesting that
 the Mikado braids satisfy a positivity property involving the
 canonical Kazhdan-Lusztig basis $\cC:= \{C_w~|~w \in W\}$ of the
 Iwahori--Hecke algebra\index{Iwahori-Hecke algebra} $H(W)$ related to
 the Coxeter system $(W,S)$, see
 \cite{a4-c3-c13-KL,a4-c3-c13-DG}. There is a natural group
 homomorphism $a: \cA_W\longrightarrow H(W)^\times$ from $\cA_W$ into
 the multiplicative group $H(W)^\times$ of $H(W)$.  The image of a
 Mikado braid, that is of a rational permutation braid, in
 $H(W)^\times$ has as coefficients Laurent polynomials with
 non-negative coefficients when expressed in the canonical basis $\cC$
 by a result by Dyer and Lehrer (see
 \cite{a4-c3-c13-DL,a4-c3-c13-DG}).
 
 \section{The Hurwitz action}
\textit{Hurwitz action in Coxeter systems.}  Deligne showed the dual
 Matsumoto property in spherical Coxeter systems, that is he showed that
 the Hurwitz action of the braid group $\cB_{l_T(c)}$ on
 $\operatorname{Red}_T(c)$ is transitive for every Coxeter element $c$
 in $(W,S)$ \cite{a4-c3-c13-Del2}; and Igusa and Schiffler proved it
 for arbitrary Coxeter systems \cite{a4-c3-c13-IS}.  In
 \cite{a4-c3-c13-BDSW} a new, more general and first of all
 constructive proof of this property is given:
 
 \begin{theorem}[{\cite[Thm.~1.3]{a4-c3-c13-BDSW}}]\label{a4-c3-c13-HurwitzTrans}
   Let $(W,T)$ be a (finite or infinite) dual Coxeter system of finite
   rank $n$ and let $c = s_1 \cdots s_m$ be a parabolic Coxeter
   element in W. The Hurwitz action on $\operatorname{Red}_T(c)$ is
   transitive.
 \end{theorem}
 
 Theorem~\ref{a4-c3-c13-HurwitzTrans} is also more general then
 Theorem~1.4 in \cite{a4-c3-c13-IS}, as in \cite{a4-c3-c13-BDSW} dual
 Coxeter systems are considered while in \cite{a4-c3-c13-IS} Coxeter
 systems, and in general the set of Coxeter elements is in a dual
 system larger than that one in a Coxeter system.

 The proof of Thereom~\ref{a4-c3-c13-HurwitzTrans} is based on a study
 of the Cayley graphs ${\rm CG}_S(W)$ and ${\rm CG}_T(W)$.  Using the
 same methods one can also show that every reflection occurring in a
 reduced $T$-factorisation of an element of a parabolic subgroup $P$
 of $W$ is already contained in that parabolic subgroup.
 
 \begin{theorem}[{\cite[Thm.~1.4]{a4-c3-c13-BDSW}}]\label{a4-c3-c13-Factor} 
   Let $(W,S)$ be a (finite or infinite) Coxeter system, $P$ a
   parabolic subgroup and $w \in P$. Then
   $\operatorname{Red}_T(w) = \operatorname{Red}_{T \cap P}(w)$.
 \end{theorem}
 
 This basic fact was not known before and can be seen as a founding
 stone towards a general theory for `dual' Coxeter systems.

 \subsection*{Hurwitz action in the spherical Coxeter systems and quasi-Coxeter elements.}
In the rest of the section, $(W,T)$ is a finite dual Coxeter system.
 
 In order to understand the dual Coxeter systems $(W,T)$ one also
 needs to know for which elements in $W$ the Hurwitz action is
 transitive.  The answer to that question is as follows
 \cite{a4-c3-a4-c3-c13-BGRW}.
 
 A \emph{parabolic quasi-Coxeter element} is an element $w \in W$ that
 has a reduced factorisation into reflections such that these
 reflections generate a parabolic subgroup of $W$.
 
 Note if one reduced $T$-factorisation of $w \in W$ generates a
 parabolic subgroup $P$ then every reduced $T$-factorisation of $w$ is
 in $P$ by Theorem~\ref{a4-c3-c13-Factor}.  It also follows that every
 such factorisation generates $P$
 \cite[Thm.~1.2]{a4-c3-a4-c3-c13-BGRW}.
 
 If a factorisation of $w$ generates the whole group $W$, it is a
 \emph{quasi-Coxeter element}\index{Coxeter element!quasi-Coxeter element}. 
Clearly every Coxeter element is a quasi-Coxeter element.
 In type $A_n$ and $B_n$ every quasi-Coxeter element is already a
 Coxeter element. The smallest Coxeter system containing a proper
 quasi-Coxeter element is of type $D_4$.
 
 Now we can answer the question above.

  \begin{theorem}[{\cite[Thm.~1.1]{a4-c3-a4-c3-c13-BGRW}}]\label{a4-c3-c13-Quasi-Cox-Lattice} 
    Let $(W,S)$ be a spherical Coxeter system and let $w \in W$. The
    Hurwitz action is transitive on $\operatorname{Red}_T(w)$ if and
    only if $w$ is a parabolic quasi-Coxeter element.
 \end{theorem}

 Recently, Wegener showed that the dual Matsumoto property holds for
 quasi-Coxeter elements in affine Coxeter systems as well
 \cite{a4-c3-c13-We}.  These two results have the following
 consequence.
 
 \begin{corollary}
   Let $(W,T)$ be a dual Coxeter system, $w\in W$ and
   $w = t_1 \cdots t_m$ a reduced $T$-factorisation, then the Hurwitz
   action is transitive on $\operatorname{Red}_T(w)$ in the Coxeter
   group $W^\prime:= \langle t_1, \ldots , t_m\rangle$ whenever
   $W^\prime$ is a spherical or an affine Coxeter group.
 \end{corollary}
 \begin{proof}
 	According to  Theorem~3.3 of \cite{a4-c3-c13-Dy},  $W^\prime:= \langle t_1, \ldots , t_m\rangle$ is a Coxeter group. Theorem~\ref{a4-c3-c13-Quasi-Cox-Lattice} and the main result in
 	\cite{a4-c3-c13-We} then yield the statement.
 \end{proof}
 
 The (parabolic) quasi-Coxeter elements are interesting for more reasons; for instance also for the following.
 Let $\Phi$ be the root system related to $(W,S)$ and let $L(\Phi) := \ZZ\Phi$ and $L(\Phi^{\vee}):= \ZZ\Phi^{\vee}$ where $\alpha^{\vee}:= 2\alpha/(\alpha, \alpha)$ be the root and the coroot lattices,  respectively.
 Quasi-Coxeter elements are also intrinsic in the dual Coxeter systems as they generate the root as well as the coroot lattice:
 Let $w = t_1 \cdots t_n$ be a reduced $T$-factorisation of $w \in W$  and let $\alpha_i \in \Phi$ be the root  related to the reflection $t_i$
 for $1 \leq i \leq n$. 
 
 \begin{theorem}[{\cite[Thm.~1.1]{a4-c3-c13-BW}}]
 	Let $\Phi$ be a finite crystallographic  root system of rank $n$. Then $w$ is a quasi-Coxeter element if and only if
 	\begin{enumerate}
 		\item $\{\alpha_i~|~1 \leq i \leq n\}$ is a $\ZZ$-basis of  the root lattice $L(\Phi)$, and 
 		\item $\{ \alpha_i^{\vee}~|~1 \leq i \leq n\}$ is a $\ZZ$-basis of the coroot lattice $L(\Phi^{\vee})$.
 	\end{enumerate}
 \end{theorem} 
 
 Thus if all the roots in $\Phi$ are of the same length, then
 $L(\Phi) = L(\Phi^{\vee})$ and the quasi-Coxeter elements correspond
 precisely to the basis of the root lattice.
 
 Quasi-Coxeter elements and Coxeter elements share further important
 properties beyond Hurwitz transitivity.
 
 \begin{theorem}[{\cite[Cor.~6.11]{a4-c3-a4-c3-c13-BGRW}}]
   An element $x \in W$ is a parabolic quasi-Coxeter element if and
   only if $x \leq_T w$ for a quasi-Coxeter element $w$.
 \end{theorem}

 Finally, Gobet observed that, in a spherical Coxeter system, every
 parabolic quasi-Coxeter element can be uniquely written as a product
 of commuting parabolic quasi-Coxeter elements \cite{a4-c3-c13-Go}.
 This factorisation of a quasi-Coxeter element can be thought of as a
 generalisation of the unique disjoint cycle decomposition of a
 permutation.

\section{Non-crossing partitions arising in representation theory}

In this section, we explain how non-crossing partitions arise naturally
in representation theory. For any finite dimensional algebra $A$ over
a field $k$ we consider the category $\operatorname{mod} A$ of finite dimensional
(right) $A$-modules and denote by $K_0(A)$ its \emph{Grothendieck
  group}\index{Grothendieck group}.  This group is free abelian of
finite rank, and a representative set of simple $A$-modules
$S_1,\ldots, S_n$ provides a basis $e_1,\ldots,e_n$ if one sets
$e_i=[S_i]$ for all $i$. As usual, we denote for any $A$-module $X$ by
$[X]$ the corresponding class in $K_0(A)$.  The Grothendieck group
comes equipped with the \emph{Euler form}\index{Euler form}
$K_0(A)\times K_0(A)\to\ZZ$ given by
\[\langle[X],[Y]\rangle=\sum_{n\ge 0}(-1)^n\dim_k\operatorname{Ext}^n_A(X,Y)\]
which is bilinear and non-degenerate (assuming that $A$ is of finite
global dimension). The corresponding symmetrised form is given by
$(x,y)=\langle x,y\rangle +\langle y,x\rangle$.  For a class $x=[X]$
given by a module $X$, one defines the reflection
\begin{equation}\label{eq:defn-reflection}
  s_x\colon K_0(A)\longrightarrow K_0(A),\quad a\mapsto
  a-2\frac{(a,x)}{(x,x)}x,
\end{equation}
assuming that $(x,x)\neq 0$ divides
$(e_i,x)$ for all $i$.  Let us denote by $W(A)$ the group of
automorphisms of $K_0(A)$ that is generated by the set of simple
reflections $S(A)=\{s_{e_1},\ldots,s_{e_n}\}$; it is called the
\emph{Weyl group}\index{Weyl group} of $A$.

From now on, assume that $A$ is \emph{hereditary}, that is, of global
dimension at most one.  Then, one can show that the Weyl group $W(A)$
is actually a Coxeter group. For example, the path algebra $kQ$ of any
quiver $Q$ is hereditary and in that case $kQ$-modules identify with
$k$-linear representations of $Q$.

\begin{proposition}[{\cite[Thm.~B.2]{a4-c3-c13-HK2016}}]
  A Coxeter system $(W,S)$ is of the form $(W(A),S(A))$ for some finite
  dimensional hereditary algebra $A$ if and only if it is
  crystallographic in the following sense:
\begin{enumerate}
\item $m_{st}\in\{2,3,4,6,\infty\}$ for all $s\neq t$ in $S$, and
\item in each circuit of the Coxeter graph not containing the edge
label $\infty$, the number of edges labelled $4$ (resp.\ $6$) is even.\qed
\end{enumerate}
\end{proposition}

We may assume that the simple $A$-modules are numbered in such a way
that $\langle e_i,e_j\rangle=0$ for $i >j$, and we set
$c=s_{e_1}\cdots s_{e_n}$. Note that $c=c(A)$ is a \emph{Coxeter
  element} which is determined by the formula
\[\langle x,y\rangle=-\langle y,c(x)\rangle\qquad\text{for}\qquad x,y\in K_0(A).\]

We are now in a position to formulate a theorem which provides an
explicit bijection between certain subcategories of $\operatorname{mod} A$ and the
non-crossing partitions in $\operatorname{NC}(W(A),c)$. Call a full subcategory
$\cC\subseteq \operatorname{mod} A$ \emph{thick} if it is closed under direct
summands and satisfies the following two-out-of-three property: any
exact sequence $0\to X\to Y\to Z\to 0$ of $A$-modules lies in
$\cC$ if two of $\{X,Y,Z\}$ are in $\cC$. A subcategory
is \emph{coreflective} if the inclusion functor admits a right adjoint.

\begin{theorem}\label{th:algebras-main}
  Let $A$ be a hereditary finite dimensional algebra.  Then, there is
  an order preserving bijection between the set of thick and
  coreflective subcategories of $\operatorname{mod} A$ (ordered by inclusion) and
  the partially ordered set of non-crossing partitions
  $\operatorname{NC}(W(A),c)$. The map sends a subcategory which is generated by an
  exceptional sequence $E=(E_1,\ldots, E_r)$ to the product of
  reflections $s_E=s_{E_1}\cdots s_{E_r}$.\qed
\end{theorem}

The rest of this article is devoted to explaining this result. In
particular, the crucial notion of an exceptional sequence will be discussed.

This result goes back to beautiful work of Ingalls and Thomas
\cite{a4-c3-c13-IT2009}.  It was then established for arbitary path algebras by
Igusa, Schiffler, and Thomas \cite{a4-c3-c13-IS}, and we refer to \cite{a4-c3-c13-HK2016}
for the general case.  Observe that path algebras of quivers cover
only the Coxeter groups of simply laced type (via the correspondence
$A\mapsto W(A)$); so there are further hereditary algebras.

We may think of Theorem~\ref{th:algebras-main} as a
\emph{categorification} of the poset of non-crossing partitions. There
is an immediate (and easy) consequence which is not obvious at all
from the original definition of non-crossing partitions; the first
(combinatorial) proof required a case by case analysis.

\begin{corollary}
  For a finite crystallographic Coxeter group, the corresponding poset
  of non-crossing partitions is a lattice.
\end{corollary}
\begin{proof}
  Any finite Coxeter group can be realised as the the Weyl group
  $W(A)$ of a hereditary algebra of finite representation type. In
  that case any thick subcategory is coreflective.  On the other hand,
  it is clear from the definition that the intersection of any
  collection of thick subcategories is again thick. This yields the
  join, but also the meet operation; so the poset of thick and
  coreflective subcategories is actually a lattice; see
  Remark~\ref{a4-c3-c13-half-a-lattice-is-a-lattice}
\end{proof}

This categorification provides some further insight into the
\emph{collection} of all posets of non-crossing partitions. This is
based on the simple observation that any thick and coreflective
subcategory $\cC\subseteq\operatorname{mod} A$ (given by an exceptional
sequence $E=(E_1,\ldots,E_r)$) is again the module category of a
finite dimensional hereditary algebra, say $\cC=\operatorname{mod} B$. Then
the inclusion $\operatorname{mod} B\to \operatorname{mod} A$ induces not only an inclusion
$K_0(B)\to K_0(A)$, but also an inclusion $W(B)\to W(A)$ for the
corresponding Weyl groups, which identifies $W(B)$ with the subgroup
of $W(A)$ generated by $s_{E_1},\ldots,s_{E_r}$, and identifies the
Coxeter element $c(B)$ with the non-crossing partition $s_E$ in
$W(A)$. Moreover, the inclusion $W(B)\to W(A)$ induces an isomorphism
\[\operatorname{NC}(W(B),c(B))\stackrel{\sim}\to \{x\in \operatorname{NC}(W(A),c(A))\mid x\le s_E\}.\]

The following result summarises this discussion; it reflects the fact
that there is a \emph{category of non-crossing partitions}. This means
that we consider a poset of non-crossing partitions not as a single
object but look instead at the relation with other posets of
non-crossing partitions.

\begin{corollary}[{\cite[Cor.~5.8]{a4-c3-c13-HK2016}}]
 \pushQED{\qed} Let $\operatorname{NC}(W,c)$ be the poset of non-crossing partitions given by a
  crystallographic Coxeter group $W$. Then, any element $x\in \operatorname{NC}(W,c)$
  is the Coxeter element of a subgroup $W'\le W$ that is again a crystallographic 
Coxeter group. Moreover,
  \[\operatorname{NC}(W',x)= \{y\in \operatorname{NC}(W,c)\mid y\le x\}.\qedhere\]
\end{corollary}

\section{Generalised Cartan lattices}

Coxeter groups and non-crossing partitions are closely related to root
systems. The approach via representation theory provides a natural
setting, because the Grothendieck group equipped with the Euler form
determines a root system; we call this a \emph{generalised Cartan
  lattice}\index{generalised Cartan lattice} and refer to
\cite{a4-c3-c13-HK2016} for a detailed study.

The following definition formalises the properties of the Grothendieck
group $K_0(A)$. A \emph{generalised Cartan lattice} is a free abelian
group ${\Gamma}\cong\ZZ^n$ with an ordered standard basis $e_1,\ldots,e_n$
and a bilinear form $\langle -,-\rangle\colon{\Gamma}\times{\Gamma}\to\ZZ$
satisfying the following conditions.
\begin{enumerate}
\item $\langle e_i,e_i\rangle>0$ and $\langle e_i,e_i\rangle$ divides $\langle e_i,e_j\rangle$ for all $i,j$. 
\item $\langle e_i,e_j\rangle =0$ for all $i>j$. 
\item $\langle e_i,e_j\rangle \le 0$ for all $i<j$. 
\end{enumerate}
The corresponding \emph{symmetrised form} is
\[ (x,y) = \langle x,y\rangle+\langle y,x\rangle\quad\text{for }
  x,y\in {\Gamma}. \] The ordering of the basis yields the \emph{Coxeter
  element}
\[\operatorname{cox}({\Gamma}):=s_{e_1}\cdots s_{e_n}.\]
We can define reflections $s_x$ as in \eqref{eq:defn-reflection} and
denote by $W=W({\Gamma})$ the corresponding \emph{Weyl group}\index{Weyl group}, which is the
subgroup of $\operatorname{Aut}({\Gamma})$ generated by the simple reflections
$s_{e_1},\ldots, s_{e_n}$. We write $\operatorname{NC}({\Gamma})=\operatorname{NC}(W,c)$ with
$c=\operatorname{cox}({\Gamma})$ for the poset of non-crossing partitions, and the set of
\emph{real roots} is
\[\Phi({\Gamma}) := \{ w(e_i) \mid w\in W({\Gamma}),\,1\le i\le n \}
  \subseteq{\Gamma}. \] A \emph{real exceptional
  sequence}\index{exceptional sequence} of ${\Gamma}$ is a sequence
$(x_1,\ldots,x_r)$ of elements that can be extended to a basis
$x_1,\ldots,x_n$ of ${\Gamma}$ consisting of real roots and satisfying
$\langle x_i,x_j\rangle =0$ for all $i>j$.  A \emph{morphisms}
${\Gamma}'\to{\Gamma}$ of generalised Cartan lattices is given by an isometry
(morphism of abelian groups preserving the bilinear form
$\langle-,-\rangle$) that maps the standard basis of ${\Gamma}'$ to a real
exceptional sequence of ${\Gamma}$. This yields a \emph{category of
  generalised Cartan lattices}.

What is this category good for?  One of the basic principles of
category theory is \emph{Yoneda's lemma} which tells us that we
understand an object ${\Gamma}$ by looking at the representable functor
$\operatorname{Hom}(-,{\Gamma})$ which records all morphisms that are received by ${\Gamma}$. In
our category all morphisms are monomorphisms, so $\operatorname{Hom}(-,{\Gamma})$  amounts
to the poset of subobjects (equivalence classes of monomorphisms
${\Gamma}'\to{\Gamma}$).

\begin{theorem}[{\cite[Thm~5.6]{a4-c3-c13-HK2016}}]
  \pushQED{\qed} The poset of subobjects of a generalised Cartan
  lattice ${\Gamma}$ is isomorphic to the poset of non-crossing partitions
  $\operatorname{NC}({\Gamma})$. The isomorphism sends a monomorphism
  $\phi\colon{\Gamma}'\to{\Gamma}$ to $s_{\phi(e_1)}\cdots s_{\phi(e_r)}$ where
  $\operatorname{cox}({\Gamma}')=s_{e_1}\cdots s_{e_r}$. Moreover, the assignment
  $w\mapsto w|_{{\Gamma}'}$ induces an isomorphism
\[W({\Gamma})\supseteq\langle s_{\phi(e_1)},\ldots,
s_{\phi(e_r)}\rangle\stackrel{\sim}\longrightarrow W({\Gamma}').\qedhere\]
\end{theorem}

\section{Braid group actions on exceptional sequences}

The link between representation theory and non-crossing partitions is
based on the notion of an exceptional sequence and the action of the
braid group on the collection of complete exceptional sequences. This
will be explained in the following section.

There are two sorts of abelian categories that we need to
consider. This follows from a theorem of Happel \cite{a4-c3-c13-Ha,a4-c3-c13-HR2002} which
we now explain.  Fix a field $k$ and consider a connected hereditary
abelian category $\cA$ that is $k$-linear with finite
dimensional Hom and Ext spaces. Suppose in addition that $\cA$
admits a \emph{tilting object}. This is by definition an object $T$ in $\cA$
with $\operatorname{Ext}^1_\cA(T,T)=0$ such that $\operatorname{Hom}_\cA(T,A)=0$
and $\operatorname{Ext}^1_\cA(T,A)=0$ imply $A=0$. Thus the functor
$\operatorname{Hom}_\cA(T,-)\colon\cA\to\operatorname{mod}\Lambda$ into the
category of modules over the endomorphism algebra
$\Lambda=\operatorname{End}_\cA(T)$ induces an equivalence
\[\mathbf{D}^b(\cA)\xrightarrow{\sim}\mathbf{D}^b(\operatorname{mod}\Lambda)\]
of derived categories \cite{a4-c3-c13-Ac3-HK2007}. There are two important classes
of such hereditary abelian categories admitting a tilting object:
module categories over hereditary algebras, and categories of coherent
sheaves on weighted projective lines in the sense of Geigle and
Lenzing \cite{a4-c3-c13-GL1987}. Happel's theorem then states that there are no
further classes.

\begin{theorem}[Happel]
  A hereditary abelian category with a tilting object is, up to a
  derived equivalence, either of the form $\operatorname{mod} A$ for some finite
  dimensional hereditary algebra $A$ or of the form $\operatorname{coh}\XX$ for
  some weighted projective line $\XX$.\qed
\end{theorem}

It is interesting to observe that these abelian categories form a
category: Any thick and coreflective subcategory is again an abelian
category of that type; so the morphisms are given by such inclusion functors.

Now, fix an abelian category $\cA$ which is either of the form
$\cA=\operatorname{mod} A$ or $\cA=\operatorname{coh}\XX$, as above. Note that in
both cases the Grothendieck group $K_0(\cA)$ is free of finite
rank and equipped with an Euler form, as explained before.  An object
$X$ in $\cA$ is called \emph{exceptional} if it is
indecomposable and $\operatorname{Ext}_\cA^1(X,X)=0$. A sequence
$(X_1,\ldots,X_r)$ of objects is called
\emph{exceptional}\index{exceptional sequence} if each
$X_i$ is exceptional and
$\operatorname{Hom}_\cA(X_i,X_j)=0=\operatorname{Ext}_\cA^1(X_i,X_j)$ for all
$i>j$. Such a sequence is \emph{complete} if $r$ equals the rank of
the Grothendieck group $K_0(\cA)$. Let $n$ denote rank of
$K_0(\cA)$. Then, the braid group $\cB_n$ on $n$
strands is acting on the collection of isomorphism classes of complete
exceptional sequences in $\cA$ via mutations, and it is an
important theorem that this action is transitive (due to
Crawley-Boevey \cite{a4-c3-c13-CB1992} and Ringel \cite{a4-c3-c13-Ri1994} for module
categories, and Kussin--Meltzer \cite{a4-c3-c13-KM2002} for coherent sheaves).

Any tilting object $T$ admits a decomposition
$T=\bigoplus_{i=1}^nT_i$ such that \linebreak $(T_1,\ldots,T_n)$ is a complete
exceptional sequence. We denote by $W(\cA)$ the group of
automorphisms of $K_0(\cA)$ that is generated by the
corresponding reflections $s_{T_1},\ldots,s_{T_n}$; it is the
\emph{Weyl group} with Coxeter element $c=s_{T_1}\cdots s_{T_n}$ and
does not depend on the choice of $T$. Thus we can consider the poset
of non-crossing partitions and we have the Hurwitz action on
factorisations of the Coxeter element as product of reflections. But
it is important to note that $W(\cA)$ is not always a Coxeter
group when $\cA=\operatorname{coh}\XX$, and it is an open question whether
the Hurwitz action is transitive.

The key observation is now the following.

\begin{proposition}
  The map
  \[(E_1,\ldots, E_r)\longmapsto s_{E_1}\cdots s_{E_r}\] which
  assigns to an exceptional sequence in $\cA$ the product of
  reflections in $W(\cA)$ is equivariant for the action of the
  braid group $\cB_r$.\qed
\end{proposition}

The proof is straightforward. But a priori it is not clear that the
product $s_{E_1}\cdots s_{E_r}$ is a non-crossing partition. In fact,
the proof of Theorem~\ref{th:algebras-main} hinges on the transitivity
of the Hurwitz action on factorisations of the Coxeter element. So the
analogue of Theorem~\ref{th:algebras-main} for categories of type
$\cA=\operatorname{coh}\XX$ remains open. A proof would provide an
interesting extension of the theory of crystallograpic Coxeter groups
and non-crossing partitions, which seems very natural in view of
Happel's theorem since the Grothendieck group $K_0(\cA)$ is a
derived invariant.

Partial results were obtained recently by Wegener in his thesis
\cite{a4-c3-a4-c3-c13-We2017}. In fact, when a weighted projective line $\XX$ is of
tubular type (that is, the weight sequence is up to permutation of the
form $(2,2,2,2), (3,3,3), (2,4,4)$ or $(2,3,6)$), then the
Grothendieck group gives rise to a tubular elliptic root system
\cite{a4-c3-c13-Sai, a4-c3-c13-STW}. Wegener showed the transitivity of the Hurwitz action
in this case. Thus, one has in particular the analogue of
Theorem~\ref{th:algebras-main} for $\operatorname{coh}\XX$ in the tubular case.


\end{document}